\documentclass[12pt,a4paper]{article}
\usepackage{geometry}
\usepackage{times,authblk,comment}
\usepackage[UKenglish]{babel}
\usepackage{amsmath, amsfonts, amssymb, amsthm, bm, stmaryrd, enumitem, mathtools}
\usepackage{graphicx, tikz-cd}
\renewcommand{\epsilon}{\varepsilon}
\renewcommand{\phi}{\varphi}
\renewcommand{\rho}{\varrho}

\newtheorem{Def}{Definition}[section]
\newenvironment{definition}{\begin{Def} \rm}{\end{Def}}
\newtheorem{lemma}[Def]{Lemma}
\newtheorem{proposition}[Def]{Proposition}

\newtheorem{theorem}[Def]{Theorem}
\newtheorem{example}[Def]{Example}
\newtheorem{remark}[Def]{Remark}

\newcommand{\komma}{,\hspace{0.3em}}
\newcommand{\id}{\text{id}}
\renewcommand{\leq}{\leqslant}
\renewcommand{\geq}{\geqslant}
\renewcommand{\emptyset}{\varnothing}

\newenvironment{smm}{\scriptsize \begin{pmatrix}}{\end{pmatrix}}
\renewcommand{\vector}[1]{{\scriptsize \begin{pmatrix} #1 \end{pmatrix}}}
\newcommand{\abs}[1]{{\left| #1 \right|}}

\newcommand{\Naturals}{{\mathbb N}}

\newcommand{\Reals}{{\mathbb R}}

\newcommand{\notperp}{\mathbin{\not\perp}}

\renewcommand{\c}{^\perp}
\newcommand{\cc}{^{\perp\perp}}

\newcommand{\ce}[1]{^{\perp_{#1}}}
\newcommand{\cce}[1]{^{\perp_{#1}\perp_{#1}}}
\newcommand{\herm}[2]{\left( #1 , #2 \right)}

\newcommand{\lin}[1]{[#1]}
\newcommand{\withoutzero}{^{\raisebox{0.2ex}{\scalebox{0.4}{$\bullet$}}}}
\newcommand{\downset}{{\downarrow}}
\newcommand{\U}{\mathbf U}
\renewcommand{\O}{\mathbf O}
\newcommand{\R}{\mathbf O}
\newcommand{\SO}{\mathbf{SO}}
\renewcommand{\R}{\mathbf R}
\newcommand{\G}{\mathbf G}
\newcommand{\Aut}{\mathbf O}
\newcommand{\similar}{\mathbin{\approx}}
\newcommand{\nsimilar}{\mathbin{\not\approx}}
\newcommand{\closure}{^{\vee}}

\begin{document}

\title{Transitivity and homogeneity of orthosets \\ and the real Hilbert spaces}

\author{Thomas Vetterlein}

\affil{\footnotesize
Institute for Mathematical Methods in Medicine and Data Based Modeling, \authorcr
Johannes Kepler University, Altenberger Stra\ss{}e 69, 4040 Linz, Austria \authorcr
{\tt Thomas.Vetterlein@jku.at}}

\date{\today}

\maketitle

\begin{abstract}\parindent0pt\parskip1ex

\noindent\vspace{-5ex}

An orthoset (also called an orthogonality space) is a set $X$ equipped with a symmetric and irreflexive binary relation $\perp$, called the orthogonality relation. In quantum physics, orthosets play a central role. In fact, a Hilbert space gives rise to an orthoset in a canonical way and can be reconstructed from it.

A complex Hilbert space can be seen as a real Hilbert space endowed with a complex structure. This fact motivates us to explore characteristic features of real Hilbert spaces by means of the abelian groups of rotations of a plane. Accordingly, we consider orthosets together with the groups of automorphisms that keep the orthogonal complement of a given pair of distinct elements fixed. We establish that, under a transitivity and a homogeneity assumption, an orthoset arises from a projective (anisotropic) Hermitian space.

To find conditions under which the latter's scalar division ring is $\Reals$ is difficult in the present framework. However, restricting considerations to divisible automorphisms, we can narrow down the possibilities to positive definite quadratic spaces over an ordered field. The further requirement that the action of these automorphisms is quasiprimitive implies that the scalar field is a subfield of~$\Reals$. 



\mbox{}\vspace{-2ex}

\end{abstract}

\section{Introduction}
\label{sec:Introduction}

An orthoset is a pair $(X,\perp)$, where $X$ is a set and $\perp$ is a symmetric, irreflexive binary relation on $X$. Elements $e$ and $f$ such that $e \perp f$ are called orthogonal and accordingly orthosets are also referred to as orthogonality spaces. Introduced by David Foulis and his collaborators, orthosets can be seen as an abstract version of the Hilbert space model underlying quantum physics \cite{Dac,Wlc}. Indeed, the typical example is $(P(H), \perp)$, where $P(H)$ is the collection of one-dimensional subspaces of a Hilbert space $H$ and $\perp$ is the usual orthogonality relation. Moreover, Hilbert spaces are determined in a unique way by their associated orthosets.

In the context of the foundations of quantum mechanics, orthosets have been investigated by numerous authors, see, e.g., \cite{Bru,Fin,HePu,Mac,Rod,Rum}. To establish conditions under which an orthoset originates from a Hilbert space, it seems natural to take advantage of the close relationship with lattice theory. Indeed, with any orthoset $(X,\perp)$ we may associate the complete ortholattice of orthoclosed subsets, denoted by ${\mathcal C}(X)$, and the task becomes to ensure that ${\mathcal C}(X)$ is isomorphic to the lattice of closed subspaces of a Hilbert space. But the lattice-theoretic approach to the foundations of quantum mechanics has a long tradition, going back to Birkhoff and von Neumann's seminal work \cite{BiNe}, and led to lattice-theoretical characterisations of the Hilbert space at least in the infinite-dimensional case, see, e.g., \cite{Wlb,Hol2}.

Also the present paper aims at improving our understanding of the basic quantum-physical model and we take up the idea of reducing the Hilbert space to its orthogonality relation. Lattice theory will again play a key role and we should in fact not claim that the approach adopted here differs fundamentally from previous research lines. However, we do wish to say that our point of view is in some respect uncommon. We do not view orthosets, and hence ortholattices, as the central entity around which the remaining structure is built. We rather assign to orthosets the role of underlying sets of transformation groups, the group action being required to respect orthogonality. The elements of the orthosets are not really assigned any specific meaning. What rather matters in our eyes is the notion of ``change'' represented by the action of a group. The orthogonality relation furthermore prescribes which actions can be combined and might be thought of as expressing the ``independence of changes''. Accordingly, the complex Hilbert space is not thought of as arising from a lattice of propositions, but as an entity describing ``changes'' in accordance with given independence demands.

These considerations have motivated us to explore, not the complex but, the projective real Hilbert space and its groups of simple rotations. The latter are meant to be the groups of rotations of some two-dimensional subspace. From the intuitive point of view, matters simplify considerably when compared to the complex case. At least in the finite-dimensional case, a real Hilbert space can conveniently be conceived as an $n$-sphere, opposite points being identified. Moreover, an $n$-sphere has the intuitively obvious property to allow continuous transitions of some point to another one, leaving the points orthogonal to the starting and destination points fixed. Although complex Hilbert spaces will not be discussed here, it should be clear that they can be dealt with in the present framework as well, the approach being suitably extended. At least in principle, all what we need to add is the requirement that the considered symmetries are compatible with a complex structure.

Let us summarise our procedure. Let $\Aut(X)$ be the group of symmetries of an orthoset $(X,\perp)$. For each pair of distinct elements $e$ and $f$ of $X$, we consider the subgroup $G_{ef}$ of $\Aut(X)$ that consists of the automorphisms leaving all elements orthogonal to $e$ and $f$ fixed. We require, first, the transitivity of $(X,\perp)$: some element of $G_{ef}$ should map $e$ to $f$. We require, second, the homogeneity of $(X,\perp)$: any two subgroups $G_{ef}$ and $G_{e'f'}$, where $e \neq f$ and $e' \neq f'$, are conjugate via an automorphism mapping $e$ to $e'$. It turns out that these conditions are already sufficient to ensure that there is a Hermitian space $H$ such that $(X,\perp)$ is isomorphic to $(P(H),\perp)$.


In a further step, we consider, instead of the whole group $G_{ef}$, the set $R_{ef}$ of automorphisms divisible in $G_{ef}$. Requiring $R_{ef}$, for any distinct $e$ and $f$, to be an abelian subgroup of $\Aut(X)$ and making similar assumptions as before, we again have that $(X,\perp)$ is representable by means of a Hermitian space. Now, however, we can say much more: we can show that the scalar division ring is a field (i.e., commutative), endowed with the identity involution, and formally real. In fact, our refined representation theorem is based on positive-definite quadratic spaces over ordered fields.

Inner-product spaces of this kind might resemble to a good extent the Hilbert spaces over $\Reals$. However, it might be illusionary to expect that, in the framework considered here, there are natural conditions ensuring that the scalar field of a quadratic space actually coincides with $\Reals$. Our concerns relate in particular to the finite-dimensional case. However, we shall consider the following condition, fulfilled in any real Hilbert space $H$ of dimension $\geq 3$: if $U$ is a non-trivial simple rotation of $H$, then the conjugates of $U$ generate a subgroup of the orthogonal group that acts on $P(H)$ transitively. In other words, the group generated by the simple rotations acts quasiprimitively on $P(H)$; see, e.g., \cite{Prae}. Adding this condition, we have that the field of scalars does not contain non-zero infinitesimals and is thus a subfield of $\Reals$. We note that the quasiprimitivity of a group action is implied by its primitivity, a property that we could have employed alternatively; cf.\ \cite{Vet1}.

The paper is structured as follows. Section \ref{sec:atom-spaces} addresses the correspondence between orthosets and ortholattices. We include a discussion of the more general question in which way and under which conditions complete atomistic lattices may be represented by means of a suitable structure on the collection of their atoms. We consider first the case of lattices, then the case of ortholattices. In Section \ref{sec:subspace-lattices}, we deal with the correspondence between lattices and linear spaces. Again, we first review the case of lattices, then the case of ortholattices. In the latter case, a well-known theorem characterises the lattice of subspaces of an (anisotropic) Hermitian space as a complete, irreducible, AC ortholattice. We present a modified version of this theorem, because we need a formulation that involves properties of the finite lattice elements only. Finally, it should be noted that the correspondences discussed in Sections \ref{sec:atom-spaces} and \ref{sec:subspace-lattices} can not easily be extended to a categorical framework. The situation is more transparent when restricting to automorphisms and in Section \ref{sec:automorphisms} we recall shortly some relevant facts.

The lengthy preparations are the basis for tackling our actual aim. In Section \ref{sec:homogeneous-transitivity}, we characterise Hermitian spaces by means of groups acting on orthosets. The relevant orthosets are called homogenously transitive. We note that we generally assume the orthosets to have rank $\geq 4$, but otherwise the rank is not presupposed to be either finite or infinite. In Section \ref{sec:divisible-transitivity}, we refine our results in the sense that we are more specific about the scalar division ring. We show that so-called divisibly transitive orthosets are associated with positive-definite quadratic spaces. Finally, in Section \ref{sec:infinitesimals}, we consider what we call the rotation group of a divisibly transitive orthoset $(X,\perp)$. Under the hypothesis that this group acts quasiprimitively on $X$, we show that $(X,\perp)$ corresponds to a quadratic space over a subfield of the reals.

\section{Atomistic lattices and their atom spaces}
\label{sec:atom-spaces}

A central issue in this paper is the interplay between orthosets on the one hand and inner-product spaces on the other hand. We may say that lattices act as a ``mediator'' between these two sorts of structures. Indeed, orthosets lead to ortholattices, and ortholattices of a certain kind are associated with inner-product spaces.

In this section, we shall compile basic definitions and facts concerning the former issue. In view of the needs of subsequent considerations, we adopt, however, a wider perspective: we start by discussing lattices (without an orthocomplementation) and their atom spaces, and we turn afterwards to ortholattices and their associated orthosets. For any further details, we refer the reader to Maeda and Maeda's monograph \cite{MaMa}.

\subsubsection*{Reconstructing lattices from their atom spaces}

A lattice $L$ is called {\it atomistic} if any element is the join of atoms. For an atomistic lattice $L$, the collection of atoms of $L$ will be denoted by ${\mathcal A}(L)$, called the {\it atom space} of $L$.

Let $L$ be a complete atomistic lattice. The question seems natural whether ${\mathcal A}(L)$ can be equipped with a suitable structure that uniquely determines $L$. Sending each $a \in L$ to $\omega(a) = \{ p \in {\mathcal A}(L) \colon p \leq a \}$, we get an order embedding of $L$ in the powerset of ${\mathcal A}(L)$. In order to answer our question, it would thus be useful to know a way of characterising the image on $\omega$.

A possibility is to define a closure operation on ${\mathcal A}(L)$; see, e.g., \cite{Ern}. For $A \subseteq {\mathcal A}(L)$, let $A\closure = \{ p \in {\mathcal A}(L) \colon p \leq \bigvee A \}$ and call $A \subseteq {\mathcal A}(L)$ {\it supclosed} if $A$ is closed w.r.t.~$\closure$, that is, if $A\closure = A$. Ordered by set-theoretic inclusion, the collection ${\mathcal C}({\mathcal A}(L))$ of supclosed subsets of ${\mathcal A}(L)$ is a complete lattice: for $A_\iota \in {\mathcal C}({\mathcal A}(L))$, $\iota \in I$, we have that $\bigcap_\iota A_\iota$ is the infimum in ${\mathcal C}({\mathcal A}(L))$ and $(\bigcup_\iota A_\iota)\closure$ is the supremum. This lattice is isomorphic to $L$. Indeed, we easily confirm that $\omega$ establishes an isomorphism between $L$ and ${\mathcal C}({\mathcal A}(L))$. 

In general, however, this observation will not help us to reduce $L$ to a simpler structure. But we may check to which extent the finitary version of $\closure$ can describe $L$. For $A \subseteq {\mathcal A}(L)$, let $A^- = \{ q \in {\mathcal A}(L) \colon q \leq p_1 \vee \ldots \vee p_k \text{ for some } p_1, \ldots, p_k \in A \}$. A set closed w.r.t.\ $^-$ is called a {\it subspace} of $L$ and we denote the complete lattice of subspaces by ${\mathcal S}({\mathcal A}(L))$; cf., e.g., \cite[(15.1)]{MaMa}.

A complete lattice $L$ is called {\it compactly atomistic} if $L$ is atomistic and, for any atoms $r$ and $p_\iota$, $\iota \in I$, such that $r \leq \bigvee_{\iota \in I} p_\iota$, there is a finite subset $I_0 \subseteq I$ such that $r \leq \bigvee_{\iota \in I_0} p_\iota$. It is immediate that in this case, the two closure operations on ${\mathcal A}(L)$ coincide, that is, ${\closure} = {^-}$. Recall furthermore that an element $a$ of a lattice $L$ is {\it finite} if $a$ is either the bottom element or the join of finitely many atoms. The set of finite elements of $L$ is denoted by ${\mathcal F}(L)$. Clearly, ${\mathcal F}(L)$ is a join-subsemilattice of $L$.

\begin{lemma} \label{lem:LAL}
Let $L$ be a complete atomistic lattice.
\begin{itemize}

\item[\rm (i)] The complete lattice ${\mathcal S}({\mathcal A}(L))$ is compactly atomistic. Furthermore, ${\mathcal C}({\mathcal A}(L))$ is a subposet of ${\mathcal S}({\mathcal A}(L))$. In both lattices, the infima are given by set-theoretic intersection. The finite elements of both lattices coincide: the atoms are the singletons $\{p\}$, $p \in {\mathcal A}(L)$ and the finite elements are those of the form $\{ p \in {\mathcal A}(L) \colon p \leq a \}$, $a \in {\mathcal F}(L)$. The finite suprema of finite elements coincide in both lattices.

Consequently, the map
\begin{equation} \label{fml:omega}
\omega \colon L \to {\mathcal S}({\mathcal A}(L)) \komma a \mapsto \{ p \in {\mathcal A}(L) \colon p \leq a \}
\end{equation}
is an order embedding preserving arbitrary meets. Restricted to the finite elements, $\omega$ establishes an isomorphism between the join-semilattices ${\mathcal F}(L)$ and ${\mathcal F}({\mathcal S}({\mathcal A}(L)))$.

\item[\rm (ii)] Assume that $L$ is compactly atomistic. Then ${\mathcal S}({\mathcal A}(L)) = {\mathcal C}({\mathcal A}(L))$ and $\omega \colon L \to {\mathcal S}({\mathcal A}(L))$ is an isomorphism.

\end{itemize}
\end{lemma}

\begin{proof}
Cf.~\cite[(15.5)]{MaMa}.
\end{proof}

Under the assumption of compact atomisticity, a lattice $L$ can thus be described by means of its atom space ${\mathcal A}(L)$ equipped with the closure operator $^-$. In the general case, we may describe in this way at least the finite part of $L$.

We shall now go one step further and replace the closure operator $^-$, the finitary version of $\closure$, with its binary version. For any $p_1, p_2 \in {\mathcal A}(L)$, let $p_1 \star p_2 = \{ q \in {\mathcal A}(L) \colon q \leq p_1 \vee p_2 \}$. A set $P$ equipped with a map $\star \colon P \times P \to {\mathcal P}(P)$ can be made into a closure space in the obvious way: for $A \subseteq P$ we let $A^\star$ be the smallest superset of $A$ such that $p_1, p_2 \in A^\star$ implies $p_1 \star p_2 \subseteq A^\star$. A set closed w.r.t.\ $^\star$ is called {\it linear} and we denote the complete lattice of linear sets by ${\mathcal L}(P)$.

A lattice is called {\it modular} if, for each pair $a$ and $b$, we have
\begin{equation} \label{fml:modular}
(c \vee a) \wedge b = c \vee (a \wedge b) \quad \text{for any $c \leq b$.}
\end{equation}
It turns out that if an atomistic lattice $L$ is modular, then ${^-} = {^\star}$ and hence ${\mathcal S}({\mathcal A}(L)) = {\mathcal L}({\mathcal A}(L))$, that is, the subspaces of ${\mathcal A}(L)$ coincide with the linear subsets of ${\mathcal A}(L)$ \cite[(15.2)]{MaMa}. To characterise the operation $\star$, we are led to the following classical notion.

\begin{definition} \label{def:projective-space}
A {\it projective space} is a non-empty set $P$ together with a map $\star \colon P \times P \to {\mathcal P}(P)$ such that, for any $e, f, g, h \in P$, the following conditions hold.
\begin{itemize}[leftmargin=3em]

\item[(PS1)] $e, f \in e \star f$, and $e \star e = \{e\}$.

\item[(PS2)] If $g, h \in e \star f$ and $g \neq h$, then $g \star h = e \star f$.

\item[(PS3)] $e \star (f \star g) = (e \star f) \star g$.

\end{itemize}
\end{definition}

Here, we understand that $\star$ is pointwise extended to subsets, that is, for $e \in P$ and $A \subseteq P$, we put $e \star A = \bigcup \{ e \star f \colon f \in A \}$ and similarly for $A \star e$.

\begin{proposition} \label{prop:lattices-projective-spaces}
Let $L$ be a compactly atomistic, modular lattice. Then ${\mathcal A}(L)$, \linebreak equipped with the map $\star \colon {\mathcal A}(L) \times {\mathcal A}(L) \to {\mathcal P}({\mathcal A}(L)) \komma (p_1,p_2) \mapsto \{ q \colon q \leq p_1 \vee p_2 \}$, is a projective space. Moreover, $\omega \colon L \to {\mathcal L}({\mathcal A}(L)) \komma a \mapsto \{ p \in {\mathcal A}(L) \colon p \leq a \}$ is an isomorphism of lattices.

Conversely, let $(P, \star)$ be a projective space. Then ${\mathcal L}(P)$ is a compactly atomistic, modular lattice. The map $P \to {\mathcal A}({\mathcal L}(P)) \komma e \mapsto \{e\}$ is an isomorphism of projective spaces.
\end{proposition}

\begin{proof}
See \cite[(16.5),(16.3)]{MaMa}.
\end{proof}

\subsubsection*{Reconstructing ortholattices from their associated orthosets}

Let us contrast these familiar facts with the case that we deal with an atomistic lattice that comes equipped with an orthogonality relation. An {\it orthocomplementation} on a bounded lattice $L$ is an order-reversing, involutive unary operation $\c$ that maps each element $a$ to a complement of $a$. Equipped with $\c$, $L$ is called an {\it ortholattice}. Furthermore, $L$ is in this case called an {\it orthomodular lattice}, or an {\it OML} for short, if $a \leq b$ implies that there is a $c \leq a\c$ such that $b = a \vee c$.

It turns out that to equip the atom space of a complete atomistic ortholattice with a structure determining the ortholattice is straightforward. Instead of projective spaces, we use the following notion, which is of an entirely different nature \cite{Dac,Wlc}.

\begin{definition}
An {\it orthoset} (or {\it orthogonality space}) is a non-empty set $X$ equipped with a symmetric, irreflexive binary relation $\perp$, called the {\it orthogonality relation}.
\end{definition}

For a subset $A$ of an orthoset $X$, we let $A\c = \{ q \in X \colon q \perp p \text{ for all } p \in A \}$ be the {\it orthocomplement} of $A$. The map sending any $A \subseteq X$ to $A\cc$ is a closure operation and we call sets that are closed w.r.t.\ $\cc$ {\it orthoclosed}. The complete lattice of orthoclosed subsets of $X$ is denoted ${\mathcal C}(X)$ and $\c$ makes ${\mathcal C}(X)$ into a complete ortholattice. Again, the infima in ${\mathcal C}(X)$ are given by the set-theoretic intersection, and for orthoclosed subsets $A_\iota$, $\iota \in I$, we have that $\bigvee_\iota A_\iota = (\bigcup_\iota A_\iota)\cc$.

Let $L$ be a complete atomistic ortholattice. Elements $a$ and $b$ of $L$ are called {\it orthogonal} if $a \leq b\c$; we write $a \perp b$ in this case. Obviously, ${\mathcal A}(L)$ equipped with the orthogonality relation $\perp$ inherited from $L$ is an orthoset. Moreover, we readily check that ${\closure} = {\cc}$. That is, the complete lattice of orthoclosed subsets of ${\mathcal A}(L)$ coincides with the complete lattice of supclosed subsets of ${\mathcal A}(L)$. Hence it makes sense to denote both lattices by ${\mathcal C}({\mathcal A}(L))$ and we have that $\omega \colon L \to {\mathcal C}({\mathcal A}(L)) \komma a \mapsto \{ p \in {\mathcal A}(L) \colon p \leq a \}$ is a lattice isomorphism. We conclude that, without any restrictions, the orthogonality relation on ${\mathcal A}(L)$ is suitable to describe $L$.

The exact correspondence between ortholattices and orthosets is as follows. We call an orthoset $(X,\perp)$ {\it point-closed} \cite{Rod} if every singleton is closed, that is, $\{ p \}\cc = \{ p \}$ for any $p \in X$. We note that this property is equivalent to {\it strong irredundancy} \cite{PaVe1}: for any $p, q \in X$, $\{ p \}\c \subseteq \{ q \}\c$ implies $p = q$.

\begin{proposition} \label{prop:ortholattices-OSs}
Let $L$ be a complete atomistic ortholattice. Then $({\mathcal A}(L), \perp)$ is a point-closed orthoset. Moreover, $\omega \colon L \to {\mathcal C}({\mathcal A}(L)) \komma a \mapsto \{ p \in {\mathcal A}(L) \colon p \leq a \}$ is an isomorphism of ortholattices.

Conversely, let $(X,\perp)$ be a point-closed orthoset. Then ${\mathcal C}(X)$ is a complete atomistic ortholattice. The map $X \to {\mathcal A}({\mathcal C}(X)) \komma e \mapsto \{e\}$ is an isomorphism of orthosets.
\end{proposition}

\begin{proof}
We already know that $\omega \colon L \to {\mathcal C}({\mathcal A}(L))$ is an isomorphism of lattices. To see that $\omega$ preserves the orthocomplementation, let $a \in L$. We have $\omega(a)\c = \{ p \in {\mathcal A}(L) \colon p \leq a \}\c = \{ q \in {\mathcal A}(L) \colon q \perp p \text{ for all } p \in {\mathcal A}(L) \text{ such that } p \leq a \} = \{ q \in {\mathcal A}(L) \colon q \perp a \} = \omega(a\c)$.

The assertions of the second paragraph are clear.
\end{proof}

\section{Atomistic lattices and linear spaces}
\label{sec:subspace-lattices}

We now turn to the second afore-mentioned issue: the correspondence between lattices and linear spaces. Again, we begin by mentioning the case of linear spaces (without predefined orthogonality relation) and we discuss then in some detail the case of inner-product spaces.

The main result of the present section is a representation theorem for Hermitian spaces by means of its associated ortholattice, differing from the common version in that it is based on properties of the finite part of the lattice only.

\subsubsection*{Reconstructing linear spaces from their subspace lattices}

We use the shortcut {\it sfield} to refer to a skew field (i.e., to a division ring). Let $H$ be a linear space over some sfield. We generally assume linear spaces not to have less than three dimensions but we do allow the case of infinite dimensions. We denote by ${\mathcal L}(H)$ the set of subspaces of $H$, partially ordered by set-theoretic inclusion. Then ${\mathcal L}(H)$ is a complete lattice.

Assuming a dimension $\geq 4$, we may characterise $H$ by means of ${\mathcal L}(H)$. The key properties of the lattice happen to occur in Proposition \ref{prop:lattices-projective-spaces}, which describes those atomistic lattices that are determined by the relation between triples of atoms according to which the first one is below the join of the other two. There is little to add: the reducibility of the lattice as well as lengths $\leq 3$ are to be excluded. For the difficult half of the subsequent fundamental theorem, see \cite[Ch.~VII]{Bae} or \cite[(33.6)]{MaMa}.

We call a lattice $L$ {\it irreducible} if $L$ is not isomorphic to the direct product of two lattices with at least two elements.

\begin{theorem} \label{thm:subspace-lattice-linear-spaces}
Let $H$ be a linear space. Then ${\mathcal L}(H)$ is an irreducible, compactly atomistic, modular lattice.

Conversely, let $L$ be an irreducible, compactly atomistic, modular lattice of length $\geq 4$. Then there is a linear space $H$ such that $L$ is isomorphic to ${\mathcal L}(H)$.
\end{theorem}

\subsubsection*{Reconstructing Hermitian spaces from their subspace ortholattices}

We shall next see how the picture changes for linear spaces that are endowed with an orthogonality relation.

Let $H$ be a linear space. We denote by $\lin{u_1, \ldots, u_k}$ the linear span of non-zero vectors $u_1, \ldots, u_k \in H$. For any subspace $E$ of $H$, we write $E\withoutzero = E \setminus \{0\}$ and we define $P(E) = \{ \lin u \colon u \in E\withoutzero \}$, the set of one-dimensional subspaces of $E$.

Let $\perp$ be a binary relation on $P(H)$. We call $(P(H),\perp)$ an {\it orthogeometry} if \linebreak $(P(H), \perp)$ is an orthoset with the following properties: (OG1)~$\lin w \perp \lin u, \lin v$ implies $\lin w \perp \lin x$ for any $x \in \lin{u,v}\withoutzero$, and (OG2)~for any distinct $\lin u, \lin v$ there is a $w \in \lin{u,v}\withoutzero$ such that $\lin w \perp \lin u$.

We may regard $\perp$ in this case alternatively as a relation on $H$ itself: for $u, v \in H$, let $u \perp v$ if one of $u$ or $v$ is $0$ or otherwise $\lin u \perp \lin v$. Note moreover that, by (OG1), any orthoclosed subset of $P(H)$ is of the form $P(E)$ for some subspace $E$ of $H$. We call $E$ in this case orthoclosed as well and we denote the ortholattice of all orthoclosed subspaces of $H$ by ${\mathcal C}(H)$.

The notion of an orthogeometry was introduced, with a regularity assumption instead of irreflexivity, by Faure and Fr\" olicher \cite[Chapter~14.1]{FaFr2}. In our, more special case, also the term ``orthogonality'' has been used \cite{FaFr1}.

\begin{lemma} \label{lem:orthogonality-on-linear-space}
Let $H$ be a linear space over the sfield $K$ and let $(P(H), \perp)$ be an orthogeometry.
\begin{itemize}

\item[\rm (i)] Let $U$ be a finite-dimensional subspace of $H$. Then any set of mutually orthogonal non-zero vectors in $U$ can be extended to an orthogonal basis of $U$.

\item[\rm (ii)] For any $\lin u \in P(H)$ and any two-dimensional subspace $U$, there is a $v \in U\withoutzero$ such that $\lin v \perp \lin u$.

\item[\rm (iii)] $(P(H),\perp)$ is point-closed.

\item[\rm (iv)] For any $\lin u \in P(H)$, $\{ v \in H \colon v \perp u \}$ is a hyperplane of $H$.

\end{itemize}
\end{lemma}

\begin{proof}
Ad (i): Note first that, by (OG1) and the irreflexivity of $\perp$, any set of mutually orthogonal vectors is linearly independent.

Assume that, for some $k \geq 1$, there are pairwise orthogonal non-zero vectors $e_1, \ldots, \linebreak e_k \in U$ and that there is an $f \in U$ not in $\lin{e_1, \ldots, e_k}$. By (OG2), there is an $f_1 \perp e_1$ in $\lin{e_1, f}\withoutzero$; we then have $\lin{e_1,f} = \lin{e_1,f_1}$. Similarly, there is an $f_2 \perp e_2$ in $\lin{e_2,f_1}\withoutzero$; we then have $f_2 \perp e_1,e_2$ and $\lin{e_1,e_2,f} = \lin{e_1,e_2,f_1} = \lin{e_1,e_2,f_2}$. Continuing in a same manner, we conclude that there is an $f_k$ such that $\lin{e_1,\ldots,e_k,f} = \lin{e_1,\ldots,e_k,f_k}$ and $f_k \perp e_1, \ldots, e_k$. As $U$ is of finite dimension, the assertion follows.

Ad (ii): If $u \in U$, the claim holds by (OG2). Assume that $u \notin U$. Then $U + \lin u$ is $3$-dimensional. By part (i), $U + \lin u$ possesses an orthogonal basis $\{ u, e_1, e_2 \}$. In view of (OG1), any $v \in \lin{e_1, e_2}\withoutzero \cap U\withoutzero$ fulfills the requirements.

Ad (iii): Let $u \in H\withoutzero$. Clearly, $\lin u \in \{\lin u\}\cc$. Assume that $v$ is linearly independent from $u$ and $\lin v \in \{\lin u\}\cc$. By (OG2), there is a non-zero vector $w \perp u$ in $\lin{u,v}$. But then $\lin w \in \{\lin u\}\c$ and, by (OG1), $\lin w \in \{\lin u\}\cc$, in contradiction to irreflexivitiy.

Ad (iv): By (OG1) and irreflexivitiy, $\{ v \in H \colon v \perp u \}$ is a proper subspace of $H$. Moreover, by part (ii), $\{\lin u\}\c \cap P(U) \neq \emptyset$ for any two-dimensional subspace $U$.
\end{proof}

Orthogeometries give rise to inner products. The key result originates from \cite{BiNe} and deals with the finite-dimensional case. For the generalisation to infinite dimensions, which we state below, see \cite{FaFr2}.

A {\it $\star$-sfield} is an sfield $K$ equipped with an involutorial antiautomorphism $^\star \colon K \to K$. Let $H$ be a (left) linear space over the $\star$-sfield $K$. By a {\it Hermitian form} on $H$, we mean a map $\herm{\cdot}{\cdot} \colon H \times H \to K$ such that, for any $u, v, w \in H$ and $\alpha, \beta \in K$, we have
\begin{align*}
& \herm{\alpha u + \beta v}{w} \;=\; \alpha \, \herm{u}{w} + \beta \, \herm{v}{w}, \\
& \herm{w}{\alpha u + \beta v} \;=\;
                               \herm{w}{u} \, \alpha^\star + \herm{w}{v} \, \beta^\star, \\
& \herm{u}{v} \;=\; \herm{v}{u}^\star, \\
& \herm{u}{u} = 0 \text{ implies } u = 0.
\end{align*}
Endowed with a Hermitian form, we call $H$ a {\it Hermitian space}. Note that, by the last condition which is not commonly assumed, we require a Hermitian space always to be anisotropic. If the $\star$-sfield $K$ is commutative and the involution $\star$ is the identity, we call $H$ a {\it quadratic space}. In this case, we also assume $K$ to be of characteristic~$\neq 2$. 

Let $H$ be a Hermitian space. For $\lin u, \lin v \in P(H)$, we define $\lin u \perp \lin v$ if $\herm u v = 0$. It is not difficult to check that $(P(H),\perp)$ is an orthogeometry, and we call $\perp$ the orthogonality relation induced by $\herm{\cdot}{\cdot}$.

\begin{theorem} \label{thm:BiNe}
Let $H$ be a linear space over an sfield $K$ and let $H$ be of dimension $\geq 3$. Let $\perp$ be such that $(P(H),\perp)$ is an orthogeometry. Then there is an involutorial antiautomorphism $^\star$ on $K$ and a Hermitian form $\herm{\cdot}{\cdot}$ based on $^\star$ such that $\herm{\cdot}{\cdot}$ induces $\perp$.
\end{theorem}

We now turn to the characterisation of Hermitian spaces by the ortholattice of their closed subspaces.

A lattice $L$ with $0$ is said to fulfil the {\it covering property} if, for any $a \in L$ and any atom $p \in L$ such that $p \nleq a$, $a$ is covered by $a \vee p$. If this property holds only when $a$ is finite, we say that $L$ fulfils the {\it finite covering property}. An atomistic lattice with the covering property is called {\it AC}.

Moreover, the irreducibility of ortholattices is understood in the expected way: an ortholattice is {\it irreducible} if it is not isomorphic to the direct product of two ortholattices with distinct bottom and top elements. We note that the irreducibility of an ortholattice is equivalent to the irreducibility of its lattice reduct.

From Theorems \ref{thm:subspace-lattice-linear-spaces} and \ref{thm:BiNe}, a lattice-theoretic characterisation of Hermitian spaces can be proved \cite[(34.5)]{MaMa}.

\begin{theorem} \label{thm:orthomodular-space}
Let $H$ be a Hermitian space. Then ${\mathcal C}(H)$ is a complete, irreducible AC ortholattice.

Conversely, let $L$ be a complete, irreducible AC ortholattice of length $\geq 4$. Then there is a Hermitian space $H$ such that $L$ is isomorphic to ${\mathcal C}(H)$.
\end{theorem}

The remainder of this section is devoted to the presentation of a modified version of Theorem \ref{thm:orthomodular-space}. The idea is to characterise the ortholattice by sole reference to the finite elements.

With any element $u$ of an ortholattice $L$, we associate the sublattice
\[ \downset u \;=\; \{ a \in L \colon a \leq u \}. \]
Assume that we can make $\downset u$ into an ortholattice whose orthogonality relation coincides with the one inherited from $L$. Then the orthocomplementation is
\[ \ce{u} \colon \downset u \to \downset u \komma a \mapsto a\c \wedge u. \]
Indeed, let $'$ be an orthocomplementation on $\downset u$ such that, for $a, b \leq u$, we have that $a \leq b'$ if and only if $a \perp b$. Then, for any $a \in \downset u$, $a' \perp a$ implies $a' \leq a\ce{u}$, and $a\ce{u} \perp a$ implies $a\ce{u} \leq a'$.

We readily check that $\downset u$, equipped with $\ce{u}$, is an ortholattice if and only if $\ce{u}$ is involutive if and only if $(u\c,u)$ is a modular pair \cite[(29.10)]{MaMa}. Below, however, we will encounter a stronger condition, which actually ensures that $\downset u$ is even an OML.

\begin{lemma} \label{lem:interval-ortholattice-OML}
Let $L$ be an ortholattice and $u \in L$. Then the sublattice $\downset u$, endowed with $\ce{u}$, is an OML if and only if, for any $a \leq b \leq u$ there is a $c \perp a$ such that $a \vee c = b$.
\end{lemma}

\begin{proof}
The ``only if'' part is clear by the definition of orthomodularity. To see the ``if'' part, assume that the latter indicated condition holds. It suffices to show that $\downset u$, equipped with $\ce{u}$, is an ortholattice.

For any $a \leq u$, we have that $a \vee a\ce{u} = u$. Indeed, choosing $b \leq u$ such that $b \perp a$ and $a \vee b = u$, we get $u = a \vee b \leq a \vee (a\c \wedge u) = a \vee a\ce{u} \leq u$.

Furthermore, for any $a, b \leq u$ such that $a \perp b$ and $a \vee b = u$, we have that $b = a\ce{u}$. Indeed, in this case $b \leq a\c \wedge u = a\ce{u}$ and hence $a\ce{u} = b \vee r$ for some $r \perp b$. It follows $u = a \vee a\ce{u} = a \vee b \vee r = u \vee r$ and hence $r = 0$.

We conclude that $a\cce{u} = a$ for any $a \in \downset u$ and it follows that $\downset u$ is indeed an ortholattice.
\end{proof}

We remark that the following lemma is shown on the basis of arguments analogous to those used for Lemma~\ref{lem:orthogonality-on-linear-space}(i).

\begin{lemma} \label{lem:finite-element-join-atom}
Let $L$ be an atomistic ortholattice such that, for any distinct atoms $p$ and $q$, there is an atom $r \perp p$ such that $p \vee q = p \vee r$. Then the following holds.
\begin{itemize}

\item[\rm (i)] An element $a$ of $L$ is finite if and only if $a$ is the join of finitely many mutually orthogonal atoms.

\item[\rm (ii)] Let $a$ be a finite element and $p$ an atom of $L$ not below $a$. Then there is a $q \perp a$ such that $a \vee p = a \vee q$.

\item[\rm (iii)] For any finite elements $a$ and $b$ such that $a \leq b$, there is a finite element $c \perp a$ such that $b = a \vee c$.

\end{itemize}
\end{lemma}

\begin{proof}
The following claim implies both parts of the lemma:

($\star$) Let $p_1, \ldots, p_k$, $k \geq 0$, be mutually orthogonal atoms and let $q$ not below their join. Then there is an atom $r \perp p_1, \ldots, p_k$ such that $p_1 \vee \ldots \vee p_k \vee q = p_1 \vee \ldots \vee p_k \vee r$.

{\it Proof of {\rm ($\star$)}:} If $k = 0$, the assertion is trivial; assume that $k \geq 1$. By assumption, there is an atom $q_1 \perp p_1$ such that $p_1 \vee q = p_1 \vee q_1$. Similarly, there is an atom $q_2 \perp p_2$ such that $p_2 \vee q_1 = p_2 \vee q_2$. Then $q_2 \perp p_1, p_2$ and $p_1 \vee p_2 \vee q = p_1 \vee p_2 \vee q_1 = p_1 \vee p_2 \vee q_2$. Arguing repeatedly in a similar way, we see that there is an atom $r = q_k$ with the desired property.
\end{proof}

\begin{lemma} \label{lem:FL-ideal}
Let $L$ be an atomistic lattice with the finite covering property. Then ${\mathcal F}(L)$ consists of the elements of finite height and is an ideal of $L$.
\end{lemma}

\begin{proof}
See \cite[Section 8]{MaMa}.
\end{proof}

\begin{lemma} \label{lem:AC-ortholattice-of-finite-height}
Any AC ortholattice of finite length is modular.
\end{lemma}

\begin{proof}
This follows from \cite[(27.6)]{MaMa}.
%
%
\end{proof}

\begin{lemma} \label{lem:complements-of-finite-elements}
Let $L$ be an atomistic ortholattice with the finite covering property and such that, for any distinct atoms $p$ and $q$, there is an atom $r \perp p$ such that $p \vee q = p \vee r$. Then ${\mathcal F}(L)$ is a modular sublattice of $L$.
\end{lemma}

\begin{proof}
By Lemma \ref{lem:FL-ideal}, ${\mathcal F}(L)$ is an ideal, in particular a sublattice, of $L$.

Let $u$ be a finite element of $L$. Then $\downset u$ is an AC lattice of finite length. Let $a \leq b \leq u$. By Lemma~\ref{lem:finite-element-join-atom}(iii), there is a $c \perp a$ such that $a \vee c = b$. By Lemma \ref{lem:interval-ortholattice-OML}, it follows that $\downset u$ is an OML. By Lemma \ref{lem:AC-ortholattice-of-finite-height}, $\downset u$ is even modular.

Let $a, b \in {\mathcal F}(L)$. Then the pair $a, b$ fulfils the modularity condition~$(\ref{fml:modular})$ in $\downset{(a \vee b)}$ and consequently also in ${\mathcal F}(L)$. We conclude that ${\mathcal F}(L)$ is modular.
\end{proof}

We are now ready to show a version of Theorem \ref{thm:orthomodular-space} that refers to finite elements only.

\begin{theorem} \label{thm:representation-by-Hermitian-spaces-adapted}
Let $L$ be a complete atomistic ortholattice with the following properties:
\begin{itemize}

\item[\rm (H1)] $L$ fulfills the finite covering property;

\item[\rm (H2)] for any distinct atoms $p$ and $q$, there is an atom $r \perp p$ such that $p \vee q = p \vee r$;

\item[\rm (H3)] for any orthogonal atoms $p$ and $q$ there is a third atom below $p \vee q$;

\item[\rm (H4)] $L$ is of height $\geq 4$.
\end{itemize}
Then there is a $\star$-sfield $K$ and a Hermitian space $H$ over $K$ such that $L$ is isomorphic to the ortholattice ${\mathcal C}(H)$.
\end{theorem}

\begin{proof}
Recall from Section \ref{sec:atom-spaces} that ${\mathcal S}({\mathcal A}(L))$ is the complete lattice of subspaces of the atom space ${\mathcal A}(L)$. By Lemma \ref{lem:LAL}, ${\mathcal S}({\mathcal A}(L))$ is compactly atomistic. We claim that ${\mathcal S}({\mathcal A}(L))$ fulfills the covering property. Indeed, let $P \subset Q \subseteq P \vee \{p\}$ for some $P, Q \in {\mathcal S}({\mathcal A}(L))$ and $p \in {\mathcal A}(L)$. Choosing $q \in Q \setminus P$, we have that $q \leq r_1 \vee \ldots \vee r_k \vee p$ for some $r_1, \ldots, r_k \in P$, and since $q \wedge (r_1 \vee \ldots \vee r_k) = 0$, it follows from (H1) that $p \leq r_1 \vee \ldots \vee r_k \vee q$ and thus $Q = P \vee \{p\}$.

By Lemma \ref{lem:complements-of-finite-elements}, (H1) and (H2) imply that ${\mathcal F}(L)$ is a modular sublattice of $L$. Furthermore, by Lemma \ref{lem:LAL}, $\omega \colon {\mathcal F}(L) \to {\mathcal F}({\mathcal S}({\mathcal A}(L))) \komma a \mapsto \{ p \in {\mathcal A}(L) \colon p \leq a \}$ is an isomorphism of join-semilattices. As $\omega$ also preserves meets, it is actually an isomorphism of lattices. We conclude that also ${\mathcal F}({\mathcal S}({\mathcal A}(L)))$ is a modular sublattice of ${\mathcal S}({\mathcal A}(L))$. By \cite[(14.1)]{MaMa}, it follows that ${\mathcal S}({\mathcal A}(L))$ is modular.

Finally, assume that ${\mathcal S}({\mathcal A}(L))$ is reducible. By \cite[(16.6)]{MaMa}, there are two distinct atoms below whose join there is no further atom. That is, there are distinct atoms $p, q \in L$ such that there is no third atom below $p \vee q$. If $p$ and $q$ are not orthogonal, this contradicts (H2), otherwise this contradicts (H3). We conclude that ${\mathcal S}({\mathcal A}(L))$ is irreducible.

We summarise that ${\mathcal S}({\mathcal A}(L))$ is an irreducible, compactly atomistic, modular lattice. As $L$ can be order-embedded into ${\mathcal S}({\mathcal A}(L))$, (H4) implies that the length of ${\mathcal S}({\mathcal A}(L))$ is at least $4$. By Theorem \ref{thm:subspace-lattice-linear-spaces}, there is a linear space $H$ over an sfield $K$ such that ${\mathcal S}({\mathcal A}(L))$ is isomorphic to the lattice ${\mathcal L}(H)$ of subspaces of $H$. Moreover, there is an order embedding $\phi \colon L \to {\mathcal L}(H)$, which establishes an isomorphism between the respective sublattices ${\mathcal F}(L)$ and ${\mathcal F}({\mathcal L}(H))$.

In particular, $\phi$ induces a bijection between the atom space ${\mathcal A}(L)$ and the set of one-dimensional subspaces $P(H)$. Under this correspondence, we can make $P(H)$ into an orthoset. In fact, $(P(H),\perp)$ is then even an orthogeometry: (OG1) obviously holds and (OG2) follows from (H2).

We conclude that, by Theorem \ref{thm:BiNe}, there is an involutorial antiautomorphism~$^\star$ of $K$ and a Hermitian form $\herm{\cdot}{\cdot}$ on $H$ based on $^\star$ that induces the orthogonality relation on $P(H)$. As $({\mathcal A}(L), \perp)$ is isomorphic to $(P(H), \perp)$, we have that ${\mathcal C}({\mathcal A}(L))$ is isomorphic to ${\mathcal C}(P(H))$. But by Proposition \ref{prop:ortholattices-OSs}, $L$ is isomorphic to ${\mathcal C}({\mathcal A}(L))$, and we may naturally identify ${\mathcal C}(P(H))$ with ${\mathcal C}(H)$. We have hence shown that $L$ is isomorphic to ${\mathcal C}(H)$.
\end{proof}

\section{Automorphisms of orthosets, ortholattices, and Hermitian spaces}
\label{sec:automorphisms}

We have discussed in Sections \ref{sec:atom-spaces} and \ref{sec:subspace-lattices} the correlation between orthosets, ortholattices, and Hermitian spaces, but we have not considered the question whether these correspondences can be extended to include the respective structure-preserving maps. We compile in this short section the results on the topic that are relevant for us. We deal exclusively with automorphisms.

To start with, there is an obvious correspondence between the automorphism group of a complete atomistic ortholattice and the automorphism group of its associated orthoset.

An {\it automorphism} of an orthoset $(X,\perp)$ is meant to be a bijection $\phi \colon X \to X$ such that, for any $p, q \in X$, $p \perp q$ if and only if $\phi(p) \perp \phi(q)$. The group of automorphisms of $(X,\perp)$ will be denoted by $\Aut(X)$.

\begin{proposition}
Let $L$ be a complete atomistic ortholattice. The automorphisms of $L$ and the corresponding orthoset $({\mathcal A}(L), \perp)$ are in a natural one-to-one correspondence.
\end{proposition}

\begin{proof}
Any automorphism of $L$ restricts to an orthogonality-preserving bijection of ${\mathcal A}(L)$. Conversely, any orthogonality-preserving bijection of ${\mathcal A}(L)$ extends to an automorphism of $L$.
\end{proof}

Let us next consider orthosets arising from some Hermitian space $H$. The correspondence between the automorphisms of $(P(H), \perp)$ and those of $H$ is described by a suitable version of Wigner's Theorem; see, e.g.,~\cite{May}.

A linear automorphism $U$ of a Hermitian space $H$ is called {\it unitary} if $U$ preserves the Hermitian form. The group of unitary operators on $H$ is denoted by $\U(H)$ and we write $I$ for its identity. We moreover denote the centre of the scalar $\star$-sfield $K$ by $Z(K)$ and we let $U(K) = \{ \epsilon \in K \colon \epsilon \epsilon^\star = 1 \}$ be the set of unit elements of $K$.

\begin{theorem} \label{thm:Wigner}
Let $H$ be a Hermitian space over the $\star$-sfield $K$ of dimension $\geq 3$. For any unitary operator $U$, the map
\begin{equation} \label{fml:map-induced-by-unitary-operator}
P(U) \colon P(H) \to P(H) \komma \lin x \mapsto \lin{U(x)}
\end{equation}
is an automorphism of $(P(H), \perp)$. The map $P \colon \U(H) \to \Aut(P(H))$ is a homomorphism, whose kernel is $\{ \epsilon I \colon \epsilon \in Z(K) \cap U(K) \}$.

Conversely, let $\phi$ be an automorphism of $(P(H), \perp)$ and assume that there is an at least two-dimensional subspace $S$ of $H$ such that $\phi(\lin x) = \lin x$ for any $x \in S\withoutzero$. Then there is a unique unitary operator $U$ on $H$ such that $\phi = P(U)$ and $U|_S$ is the identity.
\end{theorem}

Let $(X,\perp)$ be an orthoset. We shall be interested in groups of automorphisms of $(X,\perp)$ that act so-to-say minimally, that is, we shall deal with groups keeping every\-thing orthogonal to some element and its image fixed. For any $A \subseteq X$, we define
\[ \Aut(X,A) \;=\; \{ \phi \in \Aut(X) \;\colon \text{ $\phi(e) = e$ for any $e \perp A$} \}. \]
For a Hermitian space $H$, we may observe that $\Aut(P(H),P(S))$ can be conveniently be described by a subgroup of $\U(H)$, provided that $S \in {\mathcal C}(H)$ is such that $S\c$ is of dimension $\geq 2$. Indeed, in this case, we let
\begin{equation} \label{fml:UHS}
\U(H,S) \;=\; \{ U \in \U(H) \colon U|_{S\c} = \id_{S\c} \},
\end{equation}
and we have, by Theorem \ref{thm:Wigner}, that $P \colon \U(H,S) \to \Aut(P(H),P(S))$ is an isomorphism of groups.

Assume now that $H$ is a quadratic space, that is, $K$ is a (commutative) field and $^\star$ the identity. In this case, the unitary operators are commonly called {\it orthogonal} and the unitary group, now called the {\it orthogonal group}, is denoted by $\O(H)$. For some subspace $S$ of $H$, also the group defined in (\ref{fml:UHS}) is likewise denoted by $\O(H,S)$.

Let $U \in \O(H)$ be such that, for some finite-dimensional subspace $S$, $U|_{S\c}$ is the identity. Then $S$ is an invariant subspace and $U|_S$ is an orthogonal operator on $S$. Furthermore, the determinant of $U|_S$, which is $1$ or $-1$, does not depend on $S$. Indeed, if $\tilde S$ is a further finite-dimensional subspace such that $U|_{{\tilde S}\c}$ is the identity, then $\det U|_{\tilde S} = \det U|_{S+\tilde S} = \det U|_S$. Hence it makes sense to define $\SO(H)$ as the subgroup consisting of all $U \in \O(H)$ such that, for some finite-dimensional subspace $S$ of $H$, $U|_{S\c}$ is the identity and $U|_S$ has determinant $1$. If $H$ is finite-dimensional, this is the usual special orthogonal group.

We define $\SO(H,S) = \SO(H) \cap \O(H,S)$. If $S$ is two-dimensional, we call the elements of $\SO(H,S)$ {\it simple rotations}. We claim that $\SO(H)$ is generated by the simple rotations. Indeed, $\O(H)$ is generated by the orthogonal reflections along a hyperplane; see, e.g., \cite[Theorem~6.6]{Gro}. Hence every element of $\SO(H)$ is the product of an even number of such reflections. But the product of two reflections along hyperplanes is a simple rotation and thus the claim follows.

\section{Homogeneous transitivity}
\label{sec:homogeneous-transitivity}

We now turn to the main concern of this paper: we shall consider orthosets that possess, in a certain sense, a rich set of symmetries.

For any pair of distinct elements $e$ and $f$ of $X$ we consider the subgroup of $\Aut(X)$
\begin{align*}
G_{ef} & \;=\; \Aut(X,\{e,f\}) \\
& \;=\; \{ \phi \in \Aut(X) \colon \phi(x) = x \text{ for all $x \perp e,f$} \}.
\end{align*}
Intuitively, we shall understand $G_{ef}$ as consisting of those automorphisms that move the element $e$ to an element in the direction of $f$, in a way that the elements orthogonal to the ``base point'' $e$ and the ``destination point'' $f$ are kept fixed.

Our guiding example is the following.

\begin{example} \label{ex:real-Hilbert-space-1}
Let $H$ be real Hilbert space of dimension $\geq 4$. Obviously, $H$ is a quadratic space. Let $(P(H), \perp)$ be the orthoset associated with $H$. Let $u, v \in H\withoutzero$ be linearly independent, such that $\lin u$ and $\lin v$ are distinct elements of $P(H)$. Then, according to our remarks after Theorem \ref{thm:Wigner}, we have an isomorphism $P \colon \O(H, \lin{u,v}) \to G_{\lin u \lin v}$. That is, $G_{\lin u \lin v}$ can be identified with the group of orthogonal operators on $H$ that possess the invariant subspace $\lin{u,v}$ and keep its orthogonal complement elementwise fixed.
\end{example}

We will require our orthoset $(X,\perp)$ to be transitive in the sense that, for any $e \neq f$, $G_{ef}$ actually contains an automorphism that maps $e$ to $f$. Furthermore, we will postulate the homogeneity of $(X,\perp)$, in the sense that the subgroups $G_{ef}$, $e \neq f$, of $\Aut(X)$ are pairwise conjugate via an automorphism that preserves the ``base points''. We recall that a subgroup $G_1$ of $\Aut(X)$ is said to be {\it conjugate} to a further subgroup $G_2$ via some $\tau \in \Aut(X)$ if $G_1 = \tau^{-1} \, G_2 \, \tau$.

\begin{definition}
We call an orthoset $(X,\perp)$ {\it homogeneously transitive} if, for any distinct $e, f \in X$, the following holds:
\begin{itemize}[leftmargin=3em]

\item[\rm (HT1)] There is a $\phi \in G_{ef}$ such that $\phi(e) = f$.

\item[\rm (HT2)] For any further distinct elements $e'$ and $f'$, $G_{ef}$ is conjugate to $G_{e'f'}$ via an automorphism that maps $e$ to $e'$.

\end{itemize}

\end{definition}

Let us check that our main example belongs to this kind of orthosets.

\begin{example} \label{ex:real-Hilbert-space-2}
Let again $H$ be a real Hilbert space of dimension $\geq 4$. We claim that $(P(H), \perp)$ is a homogeneously transitive orthoset. Let $u, v \in H\withoutzero$ be linearly independent. By the isomorphism $P \colon \O(H, \lin{u,v}) \to G_{\lin u \lin v}$, it is clear that {\rm (HT1)} holds. Let $u', v' \in H\withoutzero$ be a further pair of linearly independent vectors. Then there is an orthogonal operator $U$ that maps $\lin u$ to $\lin {u'}$ and $\lin{u,v}$ to $\lin{u',v'}$. Conjugating $G_{\lin {u'} \lin {v'}}$ via $P(U)$ thus gives $G_{\lin u \lin v}$. Hence $(P(H),\perp)$ fulfills also {\rm (HT2)}.
\end{example}

As the next example shows, there are homogeneously transitive orthosets of a completely different kind.

\begin{example} \label{ex:Boolean-OS}
For any set $X$, $(X,\neq)$ is an orthoset. We readily check that $(X,\neq)$ is homogeneously transitive.
\end{example}

Note that, for an orthoset of the form $(X,\neq)$ as indicated in Example \ref{ex:Boolean-OS}, ${\mathcal C}(X)$ is the Boolean algebra of all subsets of $X$. Hence we will call such orthosets {\it Boolean}. Boolean orthosets are certainly not what we are interested in. But we will show that the only remaining homogeneously transitive orthosets are those arising from Hermitian spaces in the same manner as indicated in Example \ref{ex:real-Hilbert-space-2}.

Let us now fix a homogeneously transitive orthoset $(X,\perp)$ and let us assume that $(X,\perp)$ is not Boolean. By a {\it $\perp$-set}, we mean a subset of $X$ consisting of mutually orthogonal elements. The {\it rank} of $X$ is the smallest cardinal number $\lambda$ such that any $\perp$-set is of cardinality $\leq \lambda$. We assume $X$ to have a rank of at least $4$. Our aim is to verify that Theorem~\ref{thm:representation-by-Hermitian-spaces-adapted} applies to ${\mathcal C}(X)$, the ortholattice associated with $(X,\perp)$.

We begin by showing those conditions which we know to ensure the representability of $(X,\perp)$ by means of a Hermitian space if $(X,\perp)$ has a finite rank \cite{Vet2}.

\begin{lemma} \label{lem:HT-L1-L2}
$(X,\perp)$ has the following properties.
\begin{itemize}

\item[\rm (L1)] For any distinct elements $e$ and $f$, there is a $\bar e \perp e$ such that $\{e,f\}\cc = \{e,\bar e\}\cc$.

\item[\rm (L2)] For any orthogonal elements $e$ and $f$, there is a third element $g \in \{e,f\}\cc$.

\end{itemize}
\end{lemma}

\begin{proof}
Ad (L1): Let $g, h \in X$ an arbitrary pair of orthogonal elements. As we assume $(X,\perp)$ to have a rank of at least $4$, such a pair exists. By (HT2), there is an automorphism $\tau$ of $(X,\perp)$ such that $\tau(g) = e$ and $G_{gh} = \tau^{-1} \, G_{ef} \, \tau$. Then $\tau \, G_{gh}(g) = G_{ef}(\tau(g)) = G_{ef}(e)$. By (HT1), there is a $\phi \in G_{gh}$ such that $\phi(g) = h$. Then $e = \tau(g) \perp \tau(h) = \tau\phi(g) = \tau\phi\tau^{-1}(e)$. Then $\psi =  \tau\phi\tau^{-1} \in G_{ef}$ and $\bar e = \psi(e) = \tau(h) \perp e$.

Assume $x \perp e, f$. Then $x = \psi(x) \perp \psi(e) = \bar e$. Hence $\{ e, f \}\c \subseteq \{ e, \bar e \}\c$. Conversely, assume $x \perp e, \bar e$. Then $\tau^{-1}(x) \perp g, h$. By (HT1), there is a $\chi \in G_{ef}$ such that $\chi(e) = f$. Then $\tau^{-1} \chi \tau \in G_{gh}$. We conclude $\tau^{-1}(x) = \tau^{-1} \chi \tau(\tau^{-1}(x)) \perp \tau^{-1} \chi \tau(g) = \tau^{-1} \chi(e) = \tau^{-1}(f)$, that is, $x \perp f$. Hence also $\{e, \bar e\}\c \subseteq \{e, f \}\c$. The assertion follows.

Ad (L2): Let $e'$ and $f'$ be distinct elements of $X$ such that $e' \notperp f'$. As we assume $\perp$ not to coincide with $\neq$, such a pair exists. By (HT2), there is an automorphism $\tau$ such that $\tau(e) = e'$ and $G_{ef} = \tau^{-1} \, G_{e'f'} \, \tau$. By (HT1), there is a $\phi \in G_{e'f'}$ such that $\phi(e') = f'$. Then $\tau^{-1} \phi \tau \in G_{ef}$ and $g = \tau^{-1}(f') = \tau^{-1} \phi \tau (e) \in G_{ef}(e)$. From $g \perp e$ it would follow $f' \perp e'$, and from $g = e$ it would follow $e' = \phi(e') = f'$, both a contradiction. Hence $g$ is an element of $G_{ef}(e)$ distinct from $e$ and $f$. We furthermore have that $G_{ef}(e) \perp \{e,f\}\c$. Hence $g \in \{e,f\}\cc$.
\end{proof}

The next proposition implies that we may identify $(X,\perp)$ with the orthoset associated with ${\mathcal C}(X)$.

\begin{lemma} \label{lem:C-is-atomistic}
$(X,\perp)$ is point-closed. Consequently, ${\mathcal C}(X)$ is a complete atomistic ortholattice, the atoms being $\{e\}$, $e \in X$.
\end{lemma}

\begin{proof}
To show that $\{e\}\cc = \{e\}$ for any $e \in X$, we may argue as in case of in \cite[Lemma 3.2]{Vet2}; the proof applies also without the assumption of a finite rank. For the additional assertion, see Proposition \ref{prop:ortholattices-OSs}.
\end{proof}

We now turn to the verification of conditions (H1)--(H4) of Theorem \ref{thm:representation-by-Hermitian-spaces-adapted}.

\begin{lemma} \label{lem:H2-H4}
${\mathcal C}(X)$ fulfils {\rm (H2)}, {\rm (H3)}, and {\rm (H4)}.
\end{lemma}

\begin{proof}
By property (L1) in Lemma \ref{lem:HT-L1-L2}, (H2) holds.

From property (L2) in Lemma \ref{lem:HT-L1-L2}, we conclude that below the join of two orthogonal atoms of ${\mathcal C}(X)$ there is a third atom. This means that ${\mathcal C}(X)$ fulfills (H3).

As we have assumed $(X,\perp)$ to be of rank $\geq 4$, ${\mathcal C}(X)$ contains at least $4$ mutually orthogonal elements and hence a $5$-element chain. We conclude that also (H4) holds in ${\mathcal C}(X)$.
\end{proof}

It remains to check (H1), the finite covering property. Some preparatory steps are necessary.

\begin{lemma} \label{lem:automorphisms-between-finite-elements}
Let $e_1, \ldots, e_k$ and $f_1, \ldots, f_k$ be pairwise orthogonal elements of $X$, respectively. Then there is an automorphism $\phi$ of $X$ such that $\phi(e_1) = f_1, \ldots \phi(e_k) = f_k$ and $\phi(x) = x$ for any $x \perp e_1, \ldots, e_k, f_1, \ldots, f_k$.
\end{lemma}

\begin{proof}
We proceed by induction over $k$. For $k = 1$, the assertion holds by (HT1). Let $k \geq 2$ and assume the assertion holds for any two sets of $k-1$ mututally orthogonal elements. Let $e_1, \ldots, e_k$ and $f_1, \ldots, f_k$ be pairwise orthogonal, respectively. We have to show that, for some automorphism $\phi$, $\,\phi(e_1) = f_1, \ldots, \phi(e_k) = f_k$ and $\phi(x) = x$ for any $x \perp e_1, \ldots, e_k, f_1, \ldots, f_k$.

By (HT1), there is a $\phi'$ such that $\phi'(e_1) = f_1$ and $\phi'(x) = x$ for any $x \perp e_1, f_1$. Then $\phi'(e_2), \ldots, \phi'(e_k)$ and $f_2, \ldots, f_k$ are pairwise orthogonal, respectively. By assumption there is an automorphism $\tilde\phi$ such that $\tilde\phi(\phi'(e_2)) = f_2, \ldots, \tilde\phi(\phi'(e_k)) = f_k$ and $\tilde\phi(x) = x$ for $x \perp \phi'(e_2), \ldots, \phi'(e_k), f_2, \ldots, f_k$. Then $\tilde\phi \circ \phi'$ fulfils the requirements. Indeed, we have $\tilde\phi(\phi'(e_1)) = \tilde\phi(f_1) = f_1$, and $x \perp e_1, \ldots, e_k, f_1, \ldots, f_k$ implies $x = \phi'(x) \perp \phi'(e_2), \ldots \phi'(e_k)$ and hence $\tilde\phi(\phi'(x)) = \tilde\phi(x) = x$.
\end{proof}

\begin{lemma} \label{lem:finite-height}
Let $e_1, \ldots, e_k$, $k \geq 1$, be mutually orthogonal elements of $X$, let $A = \{ e_1, \ldots, e_k \}\cc$, and let $Q \subseteq A$ be a $\perp$-set. Then $Q$ contains at most $k$ elements. Moreover, $A = Q\cc$ if and only if $Q$ contains exactly $k$ elements.
\end{lemma}

\begin{proof}
Let $f_1, \ldots, f_k$ be mutually orthogonal elements of $A$. We shall show that $A = \{ f_1, \ldots, f_k \}\cc$. Both assertions will then follow.

Let $B = \{ f_1, \ldots, f_k \}\cc$. By Lemma \ref{lem:automorphisms-between-finite-elements}, there is an automorphism $\phi$ of $X$ such that $\phi(e_1) = f_1, \ldots, \phi(e_k) = f_k$ and $\phi(x) = x$ for $x \in A\c$. By Proposition \ref{prop:ortholattices-OSs}, $\phi$ induces an automorphism of ${\mathcal C}(X)$. Hence $\phi(A\c) = A\c$ implies $\phi(A) = A$, and we have $B \subseteq A = \phi(A) = \phi(\{e_1\}\cc \vee \ldots \vee \{e_k\}\cc) = \{f_1\}\cc \vee \ldots \vee \{f_k\}\cc = B$, that is, $A = B$ as asserted.
\end{proof}

\begin{lemma} \label{lem:finite-CP}
${\mathcal C}(X)$ has the finite covering property, that is, ${\mathcal C}(X)$ fulfills {\rm (H1)}.
\end{lemma}

\begin{proof}
Note first that, as ${\mathcal C}(X)$ is an atomistic ortholattice fulfilling (H2), Lemma~\ref{lem:finite-element-join-atom} applies.

Let $A \in {\mathcal C}(X)$ (as a lattice element) be finite, let $e \notin A$, and let $B \in {\mathcal C}(X)$ be such that $A \subset B \subseteq A \vee \{e\}$. By Lemma \ref{lem:finite-element-join-atom}(ii), there is an $f \perp A$ such that $A \vee \{e\} = A \vee \{f\}$. Furthermore, let $g \in B \setminus A$. Again by Lemma \ref{lem:finite-element-join-atom}(ii), there is an $h \perp A$ such that $A \vee \{g\} = A \vee \{h\}$. By Lemma \ref{lem:finite-element-join-atom}(i), there is a finite $\perp$-set $Q$ such that $A = Q\cc$. Then $(Q \cup \{h\})\cc = A \vee \{h\} = A \vee \{g\} \subseteq B \subseteq A \vee \{e\} = A \vee \{f\} = (Q \cup \{f\})\cc$. By Lemma \ref{lem:finite-height} it follows $(Q \cup \{h\})\cc = (Q \cup \{f\})\cc$ and hence $B = A \vee \{e\}$.
\end{proof}

We arrive at the main result of this section.

\begin{theorem} \label{thm:homogeneously-transitive}
Let $(X,\perp)$ be a homogeneously transitive orthoset of rank $\geq 4$. Then either $(X,\perp)$ is Boolean or else there is a Hermitian space $H$, possessing a unit vector in each one-dimensional subspace, such that $(X,\perp)$ is isomorphic to $(P(H), \perp)$.
\end{theorem}

\begin{proof}
Assume that $(X,\perp)$ is not Boolean. Then, because of Lemmas \ref{lem:C-is-atomistic}, \ref{lem:H2-H4}, and \ref{lem:finite-CP}, Theorem \ref{thm:representation-by-Hermitian-spaces-adapted} is applicable: $(X,\perp)$ is isomorphic to $(P(H), \perp)$ for some Hermitian space $H$.

By Lemma \ref{lem:automorphisms-between-finite-elements} and Theorem \ref{thm:Wigner}, there is, for any vectors $u, v \in H\withoutzero$, a unitary operator mapping $\lin u$ to $\lin v$. It follows that if there is a unit vector in $H$, all one-dimensional subspaces contain a unit vector. To ensure the existence of a unit vector, we ``rescale'' the Hermitian form if necessary; see, e.g., \cite{Hol2}.
\end{proof}

We note that, although we cannot say much about the scalar $\star$-sfields of Hermitian spaces that represent homogeneously transitive orthosets according to Theorem \ref{thm:homogeneously-transitive}, it is also clear that not all $\star$-sfields are eligible.

\begin{remark}
Let $H$ be an at least two-dimensional Hermitian space over a $\star$-sfield $K$ such that each one-dimensional subspace contains a unit vector. Then $K$ has characteristic $0$; see \cite{Jon} or \cite[Lemma 25]{Vet1}.
\end{remark}

We conclude the section with some further elementary observations on homogeneously transitive orthosets that might be found interesting. In the proofs we could at some places make use of Theorem \ref{thm:homogeneously-transitive}, but we prefer to provide direct arguments.

Given an orthoset $(X,\perp)$, consider the $G_{ef}$-orbit of an element $e \in X$:
\[ G_{ef}(e) \;=\; \{ \phi(e) \colon \phi \in G_{ef} \}. \]
Then we have that $G_{ef}(e)$ and $\{e,f\}\c$ are orthogonal subsets of $X$. Moreover, $G_{ef}$ acts transitively on the former one and the latter consists of fixed points of $G_{ef}$. In general, $f$ need not be in the orbit of $e$; to ensure that $f \in G_{ef}(e)$ is the effect of condition (HT1). Under the assumption of homogenous transitivity, the pair $(G_{ef}(e),\{e,f\}\c)$ is actually a decomposition of $X$ into two constituents, in the sense that each of these sets is the orthocomplement of the other one.

\begin{lemma} \label{lem:HT-orbits}
Let $(X,\perp)$ be a homogeneously transitive orthoset and let $e \neq f$.
\begin{itemize}

\item[\rm (i)] $G_{ef}(e) = \{e,f\}\cc$. In particular, $G_{ef}$ acts transitively on $\{e,f\}\cc$.

\item[\rm (ii)] $\{e,f\}\c$ is the set of fixed points of $G_{ef}$.

\item[\rm (iii)] Let $e' \neq f'$. Then $G_{ef} = G_{e'f'}$ if and only if $G_{ef}(e) = G_{e'f'}(e')$ if and only if $\{e,f\}\c = \{e',f'\}\c$.

\end{itemize}
\end{lemma}

\begin{proof}
Ad (i): Clearly, $G_{ef}(e) \subseteq \{e,f\}\cc$. To show the reverse inclusion, let $g \in \{e,f\}\cc$ be distinct from $e$. As ${\mathcal C}(X)$ has the finite covering property, we have $\{e,f\}\cc = \{e,g\}\cc$. Furthermore, $\{e,f\}\c = \{e,g\}\c$ implies $G_{ef} = G_{eg}$, and hence, by (HT1), we have $g \in G_{ef}(e)$.

Ad (ii): Clearly, any $x \in \{e,f\}\c$ is a fixed point of $G_{ef}$. Assume that $x \notin \{e,f\}\c$ is a further fixed point. By part (i), $x \notin \{e,f\}\cc$. By virtue of (H2), there is an $\bar x \perp e,f$ such that $\{e,f,x\}\cc = \{e,f,\bar x\}\cc$. By Lemma \ref{lem:complements-of-finite-elements}, ${\mathcal F}({\mathcal C}(X))$ is a modular sublattice of ${\mathcal C}(X)$ and we conclude that there is a $g$ such that $\{x,\bar x\}\cc \cap \{e,f\}\cc = \{g\}$. As any $\phi \in G_{ef}$ extends to an automorphism of ${\mathcal C}(X)$, it follows $\phi(\{g\}) = \phi(\{x, \bar x\}\cc) \cap \phi(\{e,f\}\cc) = \{x, \bar x\}\cc \cap \{e,f\}\cc = \{g\}$. But $g \in \{e,f\}\cc$ and by part~(i), $\{e,f\}\cc$ does not contain any fixed point.

Part (iii) follows from parts (i) and (ii).
\end{proof}

We note next that we may formulate the conditions for orthosets to arise from Hermitian spaces in a slightly modified way. This version avoids the need of excluding the case of Boolean orthosets.

\begin{proposition} \label{prop:HT3alt}
An orthoset $(X,\perp)$ is homogeneously transitive and non-Boolean if and only if, for any distinct $e, f \in X$, the following conditions hold:
\begin{itemize}[leftmargin=3.4em]

\item[\rm (HT1')] There are $\phi, \psi \in G_{ef}$ such that $\phi(e) = f$ and $\psi(e) \neq e, f$.

\item[\rm (HT2')] There is an $\bar e \perp e$ such that $G_{e\bar e} = G_{ef}$.

\end{itemize}
\end{proposition}

\begin{proof}
Let $(X,\perp)$ be homogeneously transitive and not Boolean. Let $e \neq f$. By Lemma \ref{lem:HT-L1-L2}, $\{e,f\}\cc$ contains at least three elements, among which is there is an $\bar e \perp e$. Hence (HT1') and (HT2') follow from Lemma \ref{lem:HT-orbits}.

Conversely, assume that the orthoset $(X,\perp)$ fulfils (HT1') and (HT2'). Clearly, (HT1) then holds. Furthermore, $(X,\perp)$ is not Boolean. Indeed, otherwise $G_{ef}$ would contain, for any distinct $e, f \in X$, just two elements, namely, the identity and the map interchanging $e$ and $f$, whereas (HT1') implies the existence of a third map.

We claim that $(X,\perp)$ fulfils condition (L1) in Lemma \ref{lem:HT-L1-L2}. Let $e \neq f$. By (HT2'), there is a $g \perp e$ such that $G_{ef} = G_{eg}$. We shall show that $\{e,f\}\c = \{e,g\}\c$. Indeed, let $x \perp e, f$. By (HT1'), there is $\phi \in G_{eg} = G_{ef}$ such that $\phi(e) = g$. Then $x = \phi(x) \perp \phi(e) = g$ and it follows $\{e,f\}\c \subseteq \{e,g\}\c$. We argue similarly, the roles of $f$ and $g$ being interchanged, to see that also $\{e,g\}\c \subseteq \{e,f\}\c$.

Let now $e \neq f$ and $e' \neq f'$. By (L1), there is an $\bar e \perp e$ and a $\bar{e'} \perp e'$ such that $\{e,f\}\cc = \{e,\bar e\}\cc$ and $\{e',f'\}\cc = \{e',\bar{e'}\}\cc$. By virtue of (HT1), Lemma \ref{lem:automorphisms-between-finite-elements} holds for $(X,\perp)$. Hence there is an automorphism $\tau$ such that $\tau(e) = e'$ and $\tau(\bar e) = \bar{e'}$. Then $\tau(\{e,f\}\c) = \{e',f'\}\c$. It follows $\tau^{-1} \phi \tau(x) = x$ for any $\phi \in G_{e'f'}$ and $x \perp e,f$, that is, $\tau^{-1} G_{e'f'} \tau \subseteq G_{ef}$. Similarly, we have $\tau G_{ef} \tau^{-1} \subseteq G_{e'f'}$ and we conclude $G_{ef} = \tau^{-1} G_{e'f'} \tau$. (HT2) is proved.
\end{proof}

We finally provide a further reformulation, which emphasises to some extent the role of the orbits. We might observe in this case a resemblence with the axioms of projective geometry in Definition \ref{def:projective-space}. Indeed, condition (HT1'') can be considered as similar to the requirement (PS1), according to which every line contains the two points by which it is spanned. Furthermore, condition (HT2'') may be seen as a weakened form of (PS2), according to which any two points lie on a unique line. Remarkably, (PS3) or some other variant of the Pasch axiom does not occur.

\begin{proposition}
An orthoset $(X,\perp)$ is homogeneously transitive and non-Boolean if and only if the following conditions hold:
\begin{itemize}[leftmargin=3.5em]

\item[\rm (HT1'')] For any distinct $e, f \in X$, $G_{ef}(e)$ contains at least three elements, among which are $f$ as well as some $g \perp e$.

\item[\rm (HT2'')] For any $e \in X$ and any $f, g \neq e$, if $g \in G_{ef}(e)$ then $f \in G_{eg}(e)$.

\end{itemize}
\end{proposition}

\begin{proof}
Let $(X,\perp)$ be homogeneously transitive and not Boolean. Again, for $e \neq f$, Lemma \ref{lem:HT-L1-L2} implies that $\{e,f\}\cc$ contains a third elements as well as an element orthogonal to $e$. Hence (HT1'') follows from Lemma \ref{lem:HT-orbits}. Furthermore, by Lemma~\ref{lem:finite-CP}, ${\mathcal C}(X)$ has the finite covering property. Hence (HT2'') follows from Lemma \ref{lem:HT-orbits} as well.

Conversely, assume that (HT1'') and (HT2'') hold. We will derive (HT1') and (HT2'), so that the Proposition will follow by Theorem \ref{prop:HT3alt}.

(HT1') is obvious. Let $e \neq f$. By (HT1''), there is an $\bar e \in G_{ef}(e)$ such that $\bar e \perp e$. Then $\{ e,f \}\c \subseteq \{ e,\bar e \}\c$. By (HT2''), also $f \in G_{e\bar e}(e)$ holds, which means that $\{ e,\bar e \}\c \subseteq \{ e,f \}\c$. From $\{ e,f \}\c = \{ e,\bar e \}\c$ it follows that $G_{ef} = G_{e\bar e}$. (HT2') follows.
\end{proof}

\section{Divisible transitivity}
\label{sec:divisible-transitivity}

We any pair $e, f$ of distinct elements of an orthoset we have associated the group $G_{ef}$ of automorphisms that keep the elements orthogonal to $e$ and $f$ fixed. Intuitively, we have viewed $G_{ef}$ as realising the transitions from $e$ into the direction $f$. The present section is based on the idea to add the requirement that these transitions may proceed in a continuous manner. To this end, we will certainly not deal with topologies, we rather propose a divisibility condition.

For distinct elements $e$ and $f$ of an orthoset, let us consider the following collection of automorphisms of $(X,\perp)$:
\begin{align*}
R_{ef} \;=\; & \{ \phi \in G_{ef} \colon \text{ for any $k \geq 1$ there is a $\psi \in G_{ef}$ such that $\phi = \psi^k$} \} \\
\;=\; & \{ \phi \in \Aut(X) \colon \text{ for any $k \geq 1$ there is a $\psi \in \Aut(X)$} \\
& \text{ such that $\psi(x) = x$ for all $x \perp e,f$ and $\psi^k = \phi$} \}.
\end{align*}
We shall require $R_{ef}$ to be a subgroup, subjected to similar conditions as $G_{ef}$ in case of homogeneous transitivity. Our above intuitive ideas furthermore motivate us to require $R_{ef}$ to be abelian: assuming that $\phi_0 \in R_{ef}$ is an automorphism mapping $e$ to $f$ and, for any $i \geq 0$, $\phi_{i+1}$ is a square root of $\phi_i$ in $R_{ef}$, we intend to view the subgroup generated by these maps, which is abelian, as ``dense'' in $R_{ef}$.

\begin{definition}
We call an orthoset $(X,\perp)$ {\it divisibly transitive} if, for any distinct $e, f \in X$, the following holds:
\begin{itemize}[leftmargin=3em]

\item[\rm (DT0)] $R_{ef}$ is an abelian subgroup of $\Aut(X)$.

\item[\rm (DT1)] There is a $\phi \in R_{ef}$ such that $\phi(e) = f$.

\item[\rm (DT2)] For any further distinct elements $e', f' \in X$, $R_{ef}$ is conjugate to $R_{e'f'}$ via an automorphism that maps $e$ to $e'$.

\end{itemize}

\end{definition}

\begin{example} \label{ex:real-Hilbert-space-3}
For a real Hilbert space $H$ of dimension $\geq 4$, $(P(H), \perp)$ is a divisibly transitive orthoset. Indeed, let $u, v \in H\withoutzero$ be linearly independent and recall from Example \ref{ex:real-Hilbert-space-2} the isomorphism $P \colon \O(H, \lin{ u, v}) \to G_{\lin u \lin v}$. As $R_{\lin u \lin v}$ consists of those elements of $G_{\lin u \lin v}$ that possess, for each $k \geq 1$, a $k$-th root in $G_{\lin u \lin v}$, the isomorphism restricts to $P \colon \SO(H, \lin{u,v}) \to R_{\lin u \lin v}$. That is, $R_{\lin u \lin v}$ corresponds to those operators which, restricted to $\lin{u,v}$, are (proper) rotations of the plane spanned by $u$ and $v$. {\rm (DT0)} and {\rm (DT1)} are hence clear and {\rm (DT2)} holds by a similar argument as in Example \ref{ex:real-Hilbert-space-2}.
\end{example}

We have the following, provisional representation for divisibly transitive orthosets.

\begin{lemma} \label{lem:DOS-hermitian-space}
Let $(X,\perp)$ be a divisibly transitive orthoset of rank $\geq 4$. Then there is Hermitian space $H$ over some $\star$-sfield $K$, possessing a unit vector in each one-dimensional subspace, such that $(X,\perp)$ is isomorphic to $(P(H), \perp)$.
\end{lemma}

\begin{proof}
Note first that $(X,\perp)$ cannot be Boolean. Indeed, in this case $R_{ef}$, for any $e \neq f$, would consist of the identity alone, in contradiction to (DT1).

Lemmas \ref{lem:HT-L1-L2}--\ref{lem:finite-CP} and hence Theorem \ref{thm:homogeneously-transitive} hold also for divisibly transitive orthosets. To show this, we argue on the basis of (DT1) and (DT2) instead of (HT1) and (HT2), respectively.
\end{proof}

We note that Lemma \ref{lem:HT-orbits} possesses also a version for divisibly transitive orthosets.

\begin{lemma} \label{lem:DT-orbits}
Let $(X,\perp)$ be a divisibly transitive orthoset and let $e \neq f$.
\begin{itemize}

\item[\rm (i)] $R_{ef}(e) = \{e,f\}\cc$. In particular, $R_{ef}$ acts transitively on $\{e,f\}\cc$.

\item[\rm (ii)] $\{e,f\}\c$ is the set of fixed points of $R_{ef}$.

\item[\rm (iii)] Let $e' \neq f'$. Then $R_{ef} = R_{e'f'}$ if and only if $R_{ef}(e) = R_{e'f'}(e')$ if and only if $\{e,f\}\c = \{e',f'\}\c$.

\end{itemize}
\end{lemma}

\begin{proof}
Again, we may follow literally the proof of Lemma \ref{lem:HT-orbits}, which contains the analogous statement on homogeneously transitive orthosets.
\end{proof}

Furthermore, we remark that we may formulate also the axioms of divisible transitivity in a slightly different way.

\begin{proposition}
An orthoset $(X,\perp)$ is divisibly transitive if and only if, for any distinct $e, f \in X$, {\rm (DT0)}, {\rm (DT1)}, and the following holds:
\begin{itemize}[leftmargin=3.4em]

\item[\rm (DT2')] There is an $\bar e \perp e$ such that $R_{e\bar e} = R_{ef}$.

\end{itemize}
\end{proposition}

\begin{proof}
If $(X,\perp)$ is divisibly transitive, $(X,\perp)$ fulfills condition (L1) in Lemma \ref{lem:HT-L1-L2}, and (L1) implies (DT2').

Conversely, assume (DT0), (DT1), and (DT2'). To show (DT2), we may argue similarly to the second part of the proof of Proposition \ref{prop:HT3alt}.
\end{proof}

Our aim is to refine Lemma \ref{lem:DOS-hermitian-space} and to clarify which type of $\star$-sfields are suitable for the representation of divisibly transitive orthosets. For the remainder of this section, let us fix an at least $4$-dimensional Hermitian space $H$ over a $\star$-sfield $K$ such that $(P(H),\perp)$ is divisibly transitive and each $\lin u \in P(H)$ contains a unit vector.

Let $S$ be a subspace of $H$. We set
\[ \begin{split}
\R(H,S) \;=\; \{ U \in \U(H,S) \colon & \text{for each $k \geq 1$,} \\ & \text{there is a $V \in \U(H,S)$ such that $V^k = U$} \}.
\end{split} \]
Furthermore, we will say that a subgroup $\G$ of $\U(H)$ {\it acts transitively} on $P(S)$ if, for any $u, v \in S\withoutzero$, there is a $U \in \G$ such $\lin{U(u)} = \lin v$.

\begin{lemma} \label{lem:simple-rotation-group}
Let $u, v \in H$ be linearly independent.
\begin{itemize}

\item[\rm (i)] Then $P \colon \R(H, \lin{u,v}) \to R_{\lin u \lin v}$ is a group isomorphism.

\item[\rm (ii)] $\R(H, \lin{u,v})$ is an abelian subgroup of $\U(H, \lin{u,v})$ that acts transitively on $P(\lin{u,v})$.

\end{itemize}
\end{lemma}

\begin{proof}
By Theorem \ref{thm:Wigner}, $P \colon \U(H,\lin{u,v}) \to G_{\lin u \lin v}$ is a group isomorphism. Moreover, $R_{\lin u \lin v}$ consists of the divisible elements of $G_{\lin u \lin v}$ and hence of the maps $P(U)$ such that $U \in \U(H,\lin{u,v})$ and for each $k \geq 1$ there is some $V \in \U(H,\lin{u,v})$ such that $V^k = U$. Hence $R_{\lin u \lin v} = \{ P(U) \colon U \in \R(H,\lin{u,v}) \}$. That is, the isomorphism $P \colon \U(H,\lin{u,v}) \to G_{\lin u \lin v}$ restricts to a bijection $P \colon \R(H,\lin{u,v}) \to R_{\lin u \lin v}$. But by (DT0), $R_{\lin u \lin v}$ is an abelian subgroup of $G_{\lin u \lin v}$. It follows that $\R(H,\lin{u,v})$ is an abelian subgroup of $\U(H,\lin{u,v})$ and also $P \colon \R(H,\lin{u,v}) \to R_{\lin u \lin v}$ is a group isomorphism.

By (DT1), it further follows that there is a $U \in \R(H,\lin{u,v})$ such that $P(U)(\lin u) = \lin{U(u)} = \lin v$. As $\lin u, \lin v$ are arbitrary distinct elements of $P(\lin{u,v})$, we conclude that $\R(H,\lin{u,v})$ acts transitively on $P(\lin{u,v})$.
\end{proof}

\begin{lemma} \label{lem:K-is-field}
$H$ is a quadratic space.
\end{lemma}

\begin{proof}
We have to show that the involution $^\star$ on $K$ is the identity. It will then follow that $K$ is commutative and the lemma will be proved.

Let $F$ be a two-dimensional subspace of $H$ and let $\{ u, v \}$ be an orthonormal basis of $F$. We denote the vectors of $F$ by their coordinates w.r.t.\ this basis, in particular we write $u = \begin{smm} 1 \\ 0 \end{smm}$ and $v = \begin{smm} 0 \\ 1 \end{smm}$. We will likewise identify the operators $U \in \U(H,F)$ with their restriction to $F$ and write them as $2 \times 2$-matrices. Then $U = \begin{smm} \alpha & \gamma \\ \beta & \delta \end{smm} \in \U(H,F)$ if and only if $\alpha \alpha^\star + \beta \beta^\star = \gamma \gamma^\star + \delta \delta^\star = 1$ and $\alpha \gamma^\star + \beta \delta^\star = 0$.

We proceed by showing a sequence of auxiliary statements. We will frequently use the fact that, by Lemma \ref{lem:simple-rotation-group}, $\R(H,F)$ is an abelian subgroup of $\U(H,F)$ acting transitively on $P(F)$.

(a) For any $\xi \in K$, there is a $\begin{smm} \alpha & \gamma \\ \beta & \delta \end{smm} \in \R(H,F)$ such that $\xi = \beta^{-1} \alpha$.

{\it Proof of (a):} Let $\begin{smm} \alpha & \gamma \\ \beta & \delta \end{smm} \in \R(H,F)$ be such that $\lin{\begin{smm} \alpha \\ \beta \end{smm}} = \lin{\begin{smm} \alpha & \gamma \\ \beta & \delta \end{smm} \begin{smm} 1 \\ 0 \end{smm}} = \lin{\begin{smm} \xi \\ 1 \end{smm}}$. Then $\beta \neq 0$ and $\xi = \beta^{-1} \alpha$, hence (a) follows.

(b) There is an $\epsilon \in U(K)$ such that $\xi^\star = \epsilon^\star \xi \epsilon$ for any $\xi \in K$.

{\it Proof of (b):} Let $U_1 \in \R(H,F)$ be such that $\lin{U_1 \vector{1 \\ 0}} = \lin{\vector{0 \\ 1}}$. Then $U_1 = \begin{smm} 0 & \epsilon_1 \\ \epsilon_2 & 0 \end{smm}$ for some $\epsilon_1, \epsilon_2 \in U(K)$.

Let $\xi \in K$. By (a), there is a $U = \begin{smm} \alpha & \gamma \\ \beta & \delta \end{smm} \in \R(H,F)$ such that $\xi = \beta^{-1} \alpha$.  From
\[ \begin{split}
0 & \;=\; \herm{\vector{1 \\ 0}}{\vector{0 \\ 1}} \;=\; \herm{U_1 U \vector{1 \\ 0}}{U_1 U \vector{0 \\ 1}} \;=\; \herm{U_1 U \vector{1 \\ 0}}{U U_1 \vector{0 \\ 1}} \\
& \;=\; \herm{\vector{\beta \epsilon_1 \\ \alpha \epsilon_2}}{\vector{\epsilon_1 \alpha \\ \epsilon_1 \beta}}
\;=\; \beta \epsilon_1 \alpha^\star \epsilon_1^\star + \alpha \epsilon_2 \beta^\star \epsilon_1^\star,
\end{split} \]
we conclude
$(\beta^{-1} \alpha)^\star = -\epsilon_1^\star \beta^{-1} \alpha \epsilon_2$, that is, $\xi^\star = -\epsilon_1^\star \xi \epsilon_2$.

From the case $\xi = 1$ we see that $\epsilon_2 = -\epsilon_1$. Setting $\epsilon = -\epsilon_1 = \epsilon_2$, we infer (b).

(c) Any operator in $\R(H,F)$ is of the form $\begin{smm} \alpha & -\beta^\star \\ \beta & \alpha^\star \end{smm}$ for some $\alpha, \beta \in K$ such that $\alpha \alpha^\star + \beta \beta^\star = 1$.

{\it Proof of (c):} Let $U = \begin{smm} \alpha & \gamma \\ \beta & \delta \end{smm} \in \R(H,F)$ and let $U_1 = \begin{smm} 0 & -\epsilon \\ \epsilon & 0 \end{smm}$ be as in the proof of~(b). Then
\[ \vector{\epsilon \gamma \\ \epsilon \delta}
\;=\; U U_1 \vector{1 \\ 0}
\;=\; U_1 U \vector{1 \\ 0}
\;=\; \vector{-\beta \epsilon \\ \alpha \epsilon} \]
and hence $\gamma = - \epsilon^\star \beta \epsilon = -\beta^\star$ and $\delta = \epsilon^\star \alpha \epsilon = \alpha^\star$. This shows (c).

(d) $\xi^\star = \xi$ for any $\xi \in K$.

{\it Proof of (d):} Let $U_2 \in \R(H,F)$ be such that $\lin{U_2(\vector{1 \\ 0})} = \lin{\vector{1 \\ 1}}$. Then there is a $\gamma \in K$ such that $U_2 = \begin{smm} \gamma & -\gamma^\star \\ \gamma & \gamma^\star \end{smm}$. Clearly, $\gamma \neq 0$.

Let $\xi \in K$. By (a) and (c), there is a $U = \begin{smm} \alpha & -\beta^\star \\ \beta & \alpha^\star \end{smm} \in \R(H,F)$ such that $\xi = \beta^{-1} \alpha$. From $U_2 \, U = U \, U_2$ and (c), we get
\begin{equation} \label{fml:lem:K-is-field} \begin{split}
& \alpha \gamma - \beta \gamma^\star \;=\;
  \gamma \alpha - \gamma \beta^\star \;=\;
  \gamma \alpha - \gamma^\star \beta \;=\;
  \alpha \gamma - \beta^\star \gamma, \\
& \alpha \gamma + \beta \gamma^\star \;=\;
  \gamma \alpha^\star + \gamma \beta \;=\;
\gamma \alpha + \gamma^\star \beta \;=\; 
  \alpha^\star \gamma + \beta \gamma.
\end{split} \end{equation}
We conclude $2 \alpha \gamma = 2 \gamma \alpha$ and $2 \beta \gamma^\star = 2 \gamma^\star \beta$. Hence $\gamma$ commutes with $\alpha$ and, because $2\gamma^\star = \gamma^{-1}$, also with $\beta$. Hence $\gamma$ commutes with $\xi$ and we conclude that $\gamma \in Z(K)$. From (b), it further follows that $\gamma^\star = \gamma$. It is now clear from (\ref{fml:lem:K-is-field}) that $\alpha = \alpha^\star$ and $\beta = \beta^\star$.

From (b), it follows $\alpha = \alpha^\star = \epsilon^\star \alpha \epsilon$, so that we have $\alpha \epsilon = \epsilon \alpha$. Similarly we see that $\beta \epsilon^\star = \epsilon^\star \beta$. Hence $\xi^\star = (\beta^{-1} \alpha)^\star = \epsilon^\star \beta^{-1} \alpha \epsilon = \beta^{-1} \alpha = \xi$ and the proof of (d) is complete.
\end{proof}

From the proof of the previous lemma, it is obvious how to characterise the groups $R_{\lin u \lin v}$, where $u, v \in H$ are linearly independent. We insert the following lemma for later use.

\begin{lemma} \label{lem:RHF-is-SOHF}
Let $F$ be a two-dimensional subspace of $H$. Then $\R(H,F) = \SO(H,F)$.
\end{lemma}

\begin{proof}
We have seen in the proof of Lemma \ref{lem:K-is-field} that $\R(H,F) \subseteq \SO(H,F)$.

To show the reverse inclusion, let $U \in \SO(H,F)$. Let us fix again an orthonormal basis $\{ u, v \}$ of $F$ and let us identify any operator in $\O(H,F)$ with the $2 \times 2$-matrix representing its restriction to $F$. Then $U = \begin{smm} \alpha & -\beta \\ \beta & \alpha \end{smm}$ for some $\alpha, \beta \in K$ such that $\alpha^2 + \beta^2 = 1$.

By Lemma \ref{lem:simple-rotation-group}, there is a $V \in \R(H,F)$ such that $P(V)(\lin{\vector{1 \\ 0}}) = \lin{\vector{\alpha \\ \beta}}$. Since $U(K) = \{ -1, 1 \}$, we have that $V = \begin{smm} \alpha & -\beta \\ \beta & \alpha \end{smm} = U$ or $V = \begin{smm} -\alpha & \beta \\ -\beta & -\alpha \end{smm} = -U$. We have moreover seen in the proof of Lemma \ref{lem:K-is-field} that, for some $\epsilon \in U(K)$, $\R(H,F)$ contains $\begin{smm} 0 & -\epsilon \\ \epsilon & 0 \end{smm}$ and hence also its square $\begin{smm} -\epsilon^2 & 0 \\ 0 & -\epsilon^2 \end{smm} = - \begin{smm} 1 & 0 \\ 0 & 1 \end{smm}$. We conclude that $U \in \R(H,F)$.
\end{proof}

We shall finally establish that $K$ is a formally real field. For further information on ordering on fields we refer the reader to \cite[\S 1]{Pre}.

Note that, by the next lemma, $K$ is a Pythagorean field.

\begin{lemma} \label{lem:K-is-Pythagorean}
For any $\alpha, \beta \in K$ there is a $\gamma \in K$ such that $\gamma^2 = \alpha^2 + \beta^2$. If in this case $\gamma = 0$, then $\alpha = \beta = 0$.
\end{lemma}

\begin{proof}
Let $\alpha, \beta \in K$ and assume that not both of them are equal to $0$. Let $u, v$ be orthogonal unit vectors of $H$. Then $\alpha u + \beta v$ is a non-zero vector and by the anisotropy of the Hermitian form we have $\alpha^2 + \beta^2 = \herm{\alpha u + \beta v}{\alpha u + \beta v} \neq 0$.

Moreover, $\lin{\alpha u + \beta v}$ contains a unit vector. Hence there is a $\tilde\gamma \in K \setminus \{0\}$ such that $\herm{\tilde\gamma(\alpha u + \beta v)}{\tilde\gamma(\alpha u + \beta v)} = 1$ and hence $\alpha^2 + \beta^2 = \big(\frac{1}{\tilde\gamma}\big)^2$.
\end{proof}

\begin{lemma} \label{lem:K-is-ordered}
$K$ is formally real. $K$ being equipped with any order, the Hermitian form on $H$ is positive definite.
\end{lemma}

\begin{proof}
By Lemma \ref{lem:K-is-Pythagorean}, $\alpha_1^2 + \ldots + \alpha_k^2 = 0$, where $k \geq 1$ and $\alpha_1, \ldots, \alpha_k \in K$, implies that $\alpha_1 = \ldots \alpha_k = 0$. By \cite[Theorem~(1.8)]{Pre}, it follows that $K$ is formally real.

Moreover, assume $K$ to be equipped with an order and let $u \in H\withoutzero$. Let $v \in H$ be a unit vector in $\lin u$. Then we have $u = \alpha v$ for some $\alpha \in K \setminus \{0\}$ and it follows $\herm u u = \alpha^2 > 0$.
\end{proof}

We summarise what we have shown.

\begin{theorem} \label{thm:ordered-field}
Let $(X, \perp)$ be a divisibly transitive orthoset of rank $\geq 4$. Then there is an ordered field $K$ and a positive-definite quadratic space $H$ over $K$, possessing a unit vector in each one-dimensional subspace, such that $(X, \perp)$ is isomorphic to $(P(H), \perp)$.
\end{theorem}

\section{Quasiprimitive orthosets}
\label{sec:infinitesimals}

We have dealt so far with orthosets arising from positive-definite quadratic spaces over an ordered field $K$. We certainly wonder under which natural condition on the orthoset $K$ is actually the field of real numbers. A promising way to approach this question seems, however, hard to define and we are actually not really convinced that a solution is feasible in the present framework. Here, we will rather discuss the related problem of finding reasonable conditions under which $K$ is a subfield of $\Reals$. We consider to this end divisibly transitive orthosets with a property that is once more related to transitivity.

Let $(X,\perp)$ be a divisibly transitive orthoset. We refer to an automorphism contained in a subgroup $R_{ef}$, where $e$ and $f$ are distinct elements of $X$, as a {\it simple rotation}. The identity of $R_{ef}$ is called the {\it trivial} rotation. The subgroup of $\Aut(X)$ generated by all simple rotations will be called the {\it rotation group} of $(X,\perp)$, denoted by $\R(X)$.

\begin{definition}
We call an orthoset $(X,\perp)$ {\it quasiprimitive} if, for any non-trivial simple rotation $\rho$, the normal subgroup of $\R(X)$ generated by $\rho$ acts transitively on $X$.
\end{definition}

That is, we call an orthoset quasiprimitive if the transformation group $\R(X)$ has this property; see, e.g., \cite{Prae}. More explicitly, let $\rho^\tau = \tau^{-1} \rho \tau$ be the conjugate of some automorphism $\rho$ via a further automorphism $\tau$. The quasiprimitivity of $(X,\perp)$ means that, given any simple rotation $\rho \neq \id$ and any two points $e, f \in X$, there are $\tau_1, \ldots, \tau_k \in \R(X)$ such that $f = \rho^{\tau_k} \ldots \rho^{\tau_1} (e)$.

\begin{example} \label{ex:real-Hilbert-space-4}
Let $H$ be a real Hilbert space of dimension $\geq 4$. By Example \ref{ex:real-Hilbert-space-3}, $(P(H),\perp)$ is a divisibly transitive orthoset. Moreover, the simple rotations of \linebreak $(P(H),\perp)$ are exactly the automorphisms induced by simple rotations of $H$, and $\R(P(H)) = \{ P(U) \colon U \in \SO(H) \}$.

We claim that $(P(H),\perp)$ is quasiprimitive. Let $\rho \neq \id$ be a simple rotation of $(P(H),\perp)$ and let $\lin u, \lin v \in P(H)$ be distinct. Let $S$ be a $3$-dimensional subspace of $H$ such that $u, v \in S$. $\SO(S)$, the special orthogonal group of $S$, is simple; see, e.g., \cite{Gro}. Hence the conjugates of any $U \in \SO(S)$, distinct from the identity, generate the whole group $\SO(S)$ and since $\SO(S)$ acts transitively on $P(S)$, some finite product of conjugates of $U$ maps $\lin u$ to $\lin v$. We conclude that some finite product of conjugates of $\rho$ maps $\lin u$ to $\lin v$. 
\end{example}

Let us now fix a positive-definite quadratic space $H$ over an ordered field $K$ such that each one-dimensional subspace contains a unit vector and assume that the orthoset $(P(H),\perp)$ is divisibly transitive.

We may describe the rotation group $\R(P(H))$ as follows.

\begin{proposition} \label{prop:rotations-and-SOH}
The map
\[ P \colon \SO(H) \to \R(P(H)) \]
is a surjective homomorphism.
\end{proposition}

\begin{proof}
For any linearly independent vectors $u, v \in H$, the map $P \colon \R(H, \lin{u,v}) \to R_{\lin u \lin v}$ is, by Lemma \ref{lem:simple-rotation-group}, an isomorphism. By Lemma \ref{lem:RHF-is-SOHF} and our remarks at end of Section \ref{sec:automorphisms}, the subgroups $\R(H,F)$, where $F$ is a two-dimensional subspace, generate $\SO(H)$. By definition, the subgroups $R_{\lin u \lin v}$ of $\Aut(X)$, where $\lin u, \lin v \in P(H)$ are distinct, generate $\R(P(H))$. By Theorem \ref{thm:Wigner}, $P \colon \O(H) \to \Aut(P(H))$ is a homomorphism, hence the assertion follows.
\end{proof}

The following definitions and facts are due to Holland \cite{Hol1}, for further details see also \cite{Vet1}. We call an element $\alpha \in K$ {\it infinitesimal} if $\abs{\alpha} < \tfrac 1 n$ for all $n \in \Naturals \setminus \{0\}$, and we call $\alpha \in K$ {\it finite} if $\abs{\alpha} < n$ for some $n \in \Naturals \setminus \{0\}$. We denote the set of infinitesimal and finite elements by $I_K$ and $F_K$, respectively. $I_K$ and $F_K$ are additive subgroups of $K$ and are closed under multiplication. Furthermore, $F_K \setminus I_K$ is a multiplicative subgroup of $K \setminus \{0\}$, and we have $I_K \cdot F_K = I_K$ and $F_K + I_K = F_K$.

Likewise, a vector $x \in H$ is called {\it infinitesimal} if so is $\herm x x$, and $x$ is called {\it finite} if so is $\herm x x$. The set of infinitesimal and finite vectors is denoted by $I_H$ and $F_H$, respectively. $I_H$ and $F_H$ are subgroups of $H$ and we have $I_K \cdot F_H = F_K \cdot I_H = I_H$, $\,F_K \cdot F_H = F_H$, and $F_H + I_H = F_H$. Furthermore, $\herm x y \in I_K$ if $x, y \in F_H$ and at least one of $x$ and $y$ is infinitesimal.

If the only infinitesimal element of $K$ is $0$, $K$ is called {\it Archimedean}. In this case, $K$ is isomorphic to a subfield of $\Reals$ equipped with the inherited natural order; see, e.g.,~\cite{Fuc}.

For $\lin x, \lin y \in P(H)$, we put $\lin x \similar \lin y$ if there are non-infinitesimal, finite vectors $x' \in \lin x$ and $y' \in \lin y$ such that $x' - y' \in I_H$.

\begin{lemma} \label{lem:similar}
\begin{itemize}

\item[\rm (i)] $\similar$ is an equivalence relation, which is the equality if and only if $K$ is Archimedean.

\item[\rm (ii)] For any orthogonal vectors $x, y \in H\withoutzero$, we have $\lin x \nsimilar \lin y$.

\item[\rm (iii)] The relation $\similar$ is preserved by any orthogonal operator on $H$. That is, for any $U \in \O(H)$ and $\lin x, \lin y \in P(H)$, we have $\lin x \similar \lin y$ if and only if $U(\lin x) \similar U(\lin y)$.

\end{itemize}
\end{lemma}

\begin{proof}
Ad (i): $\similar$ is clearly reflexive and symmetric. To see that $\similar$ is also transitive, let $x, y, y', z \in F_H \setminus I_H$ such that $x - y, \, y' - z \in I_H$ and $y' = \alpha y$ for some non-zero $\alpha \in K$. Then $\alpha^2 = \herm{y'}{y'} {\herm y y}^{-1} \in F_K \setminus I_K$ and hence $\alpha z \in F_H \setminus I_H$. We conclude that $x - \alpha z \in I_H$.

If $K$ is Archimedean, $x \in I_H$ implies $\herm x x = 0$ and hence $x = 0$. Hence $\similar$ is the equality. Conversely, if $K$ is not Archimedean, let $\alpha \in I_K \setminus \{0\}$ and let $u, v \in H$ be orthogonal unit vectors. Then $\lin u \similar \lin{u + \alpha v}$ because $u, u + \alpha v \in F_H \setminus I_H$ and $\alpha v \in I_H$. Hence $\similar$ is not the equality.

Ad (ii): For any finite, non-infinitesimal vectors $x' \in \lin x$ and $y' \in \lin y$, we have that $\herm{x'-y'}{x'-y'} = \herm{x'}{x'} + \herm{y'}{y'}$ is not infinitesimal because $\herm{x'}{x'}$ and $\herm{y'}{y'}$ are non-infinitesimal and positive. Hence $\lin x \nsimilar \lin y$.

Ad (iii): Assume that $\lin x \similar \lin y$. This means $x'-y' \in I_H$ for some $x' \in \lin x$ and $y' \in \lin y$ such that $x', y' \in F_H \setminus I_H$. The image of an infinitesimal vector under an orthogonal operator is obviously infinitesimal as well. Hence it follows $U(x')-U(y') \in I_H$. Furthermore, $U(x') \in \lin{U(x)} = U(\lin x)$ and similarly $U(y') \in U(\lin y)$. It is obvious again that $U(x'), U(y') \in F_H \setminus I_H$. We conclude $U(\lin x) \similar U(\lin y)$. 
\end{proof}

\begin{lemma} \label{lem:small-unitary-operator}
Assume that $K$ is non-Archimedean. Then there is a $U \in \SO(H)$ such that $P(U)$ is distinct from the identity and $U(x) - x$ is infinitesimal for any finite vector $x$.
\end{lemma}

\begin{proof}
Let $u$ and $v$ be orthogonal unit vectors and let $\alpha \in I_K \setminus \{0\}$. Furthermore, let $\epsilon \in K^+$ be such that $u' = \epsilon (u + \alpha v)$ is a unit vector. Then $\epsilon^2 = \frac 1 {1+\alpha^2} \in F_K \setminus I_K$ and hence $\epsilon \in F_K \setminus I_K$. Moreover, also $v' = \epsilon(-\alpha u + v)$ is a unit vector. In fact, $u'$ and $v'$ are orthogonal unit vectors spanning $\lin{u,v}$. Moreover, $u-u' = \alpha^2 \frac{\epsilon^2}{1+\epsilon} \, u - \alpha \epsilon \, v \in I_H$ and similarly $v-v' \in I_H$.

Let $U \in \O(H)$ be such that $U(u) = u'$, $\, U(v) = v'$, and $U(w) = w$ for any $w \perp u, v$. Then $U \in \SO(H)$ because $\det U|_{\lin{u,v}} = \epsilon^2 (1+\alpha^2) = 1$. Let $x \in H$ be finite. Then we have $x = \herm{x}{u} u + \herm{x}{v} v + w$, where $w \perp u,v$. It follows $U(x) - x = \herm{x}{u} (u'-u) + \herm{x}{v} (v'-v) \in I_H$.
\end{proof}

\begin{lemma}
If $(P(H),\perp)$ is quasiprimitive, then $K$ is Archimedean and hence a subfield of $\Reals$.
\end{lemma}

\begin{proof}
Let us assume that $(P(H),\perp)$ is quasiprimitive and $K$ is not Archimedean. In accordance with Lemma \ref{lem:small-unitary-operator}, choose some $U \in \SO(H)$ such that $P(U)$ is distinct from the identity and $U(x) - x \in I_H$ for any $x \in F_H$. Then we have that $U(\lin x) \similar \lin x$ for any $x \in H\withoutzero$.

Furthermore, for any further orthogonal operator $V$, $\,V^{-1} U V$ has the same properties. Indeed, for any $x \in F_H$, we have $\herm{V^{-1} U V(x)-x}{V^{-1} U V(x)-x} = \herm{U V(x)-V(x)}{U V(x)-V(x)} \in I_K$, that is, $V^{-1} U V(x)-x \in I_H$. Again it follows that $V^{-1} U V(\lin x) \similar \lin x$ for any $x \in H\withoutzero$.

We conclude that the orbit of any $\lin x \in P(H)$ under the action of conjugates of $U$ is contained in the $\similar$-class of $\lin x$. The latter is, by Lemma \ref{lem:similar}(ii), a proper subset of $P(H)$. In view of Proposition \ref{prop:rotations-and-SOH}, it follows that $(X,\perp)$ is not quasiprimitive.
\end{proof}

We may summarise the results of this section as follows.

\begin{theorem}
Let $(X,\perp)$ be an orthoset of rank $\geq 4$. Assume that $(X,\perp)$ is divisibly transitive and quasiprimitive. Then there is a positive-definite quadratic space $H$ over a subfield of $\Reals$, possessing a unit vector in each one-dimensional subspace, such that $(X,\perp)$ is isomorphic to $(P(H),\perp)$.
\end{theorem}

{\bf Acknowledgement.} The author acknowledges the support by the bilateral Austrian Science Fund (FWF) project I 4579-N and Czech Science Foundation (GA\v CR) project 20-09869L ``The many facets of orthomodularity''.

\end{document}